\documentclass[12pt,leqno,a4paper]{article}
\usepackage{amsmath,amsthm,amssymb}
\usepackage[utf8]{inputenc}
\usepackage[T1]{fontenc}

\usepackage{enumerate}
\usepackage{bbm}
\usepackage{todonotes}
\usepackage{mathrsfs}
\usepackage{mathabx}
\usepackage{hyperref}
\usepackage{nicefrac}
\usepackage{tikz}
\usetikzlibrary{matrix}

\usepackage{url}

\usepackage[sort,numbers]{natbib}

\providecommand{\noopsort}[1]{} 

\def\qed{\unskip\quad \hbox{\vrule\vbox
to 6pt {\hrule width 4pt\vfill\hrule}\vrule} }

\newcommand{\bez}{\nopagebreak\hspace*{\fill}
 \nolinebreak$\qed$\vspace{5mm}\par}

\newtheorem{Th}{Theorem}[section]
\newtheorem{Prop}[Th]{Proposition}
\newtheorem{Lemma}[Th]{Lemma}

\theoremstyle{definition}
\newtheorem{Remark}[Th]{Remark}

\newtheorem{Cor}[Th]{Corollary}
\newtheorem{Example}{Example}[section]

\newcommand{\beq}{\begin{equation}}
\newcommand{\eeq}{\end{equation}}

\def\scalar(#1,#2){(#1\mid#2)}

\newcommand{\cm}{{\cal M}}

\newcommand{\ot}{\otimes}
\newcommand{\ov}{\overline}
\newcommand{\la}{\lambda}

\newcommand{\R}{{\mathbb{R}}}
\newcommand{\T}{{\mathbb{T}}}
\newcommand{\C}{{\mathbb{C}}}
\newcommand{\Z}{{\mathbb{Z}}}
\newcommand{\N}{{\mathbb{N}}}

\newcommand{\PP}{{\mathbb P}}
\newcommand{\vep}{\varepsilon}

\newcommand{\lcm}{{\rm lcm}}
\newcommand{\spec}{{\rm spec}}

\title{On some classification problems of multiplicative functions}
\author{S. Kasjan, O. Klurman,  M. Lema\'nczyk}
\begin{document}
\maketitle

\begin{abstract} We prove that a multiplicative function $f:\N\to\C$ is Toeplitz if and only if there are a Dirichlet character $\chi$ and a finite subset $F$ of prime numbers such that $f(n)=\chi(n)$ for each $n$ which is coprime to all numbers from $F$. All such functions bounded by~1 are necessarily pretentious and they have exactly one Furstenberg system. Moreover, we characterize the class of pretentious functions that have precisely one Furstenberg system as those being Besicovitch (rationally) almost periodic. As a consequence, we show that the corrected Elliott's conjecture implies Frantzikinakis-Host's conjecture on the uniqueness of Furstenberg system for all real-valued bounded by~1 multiplicative functions. We also clarify relations between different classes of aperiodic multiplicative functions. \end{abstract}

\section*{Introduction}
\subsection{\label{s:Llike} Pretentious and aperiodic multiplicative functions}
In connection with a progress around the Chowla \cite{Ch} and Sarnak's conjectures \cite{Sa}, in the last decade, we observe a growing interest toward a classification of multiplicative functions\footnote{$f:\N\to\C$ is called {\em multiplicative} if $(\ast)\;f(mn)=f(m)f(n)$ for all coprime $m,n$ (if $(\ast)$ is satisfied unconditionally, then we speak about {\em completely multiplicative} functions).}  according to their ``degree of randomness''. We recall  that in the class $\cm$ of bounded by~1, multiplicative functions a special role is played by the so called multiplicative distance (introduced by Granville and Soundararajan \cite{Gr-Sa}): for $f,g\in\mathcal{M}$, we set
\beq\label{gs1}\mathbb{D}(f,g)^2=\sum_{p\in\mathbb{P}}\frac1p\Big(1-{\rm Re}(f(p)\ov{g(p)})\Big)\eeq
and we say that $f$ {\em pretends to be} $g$, denoted by $f\sim g$,  whenever $\mathbb{D}(f,g)<+\infty$. Then $f$ is called {\em pretentious} if \beq\label{pret1}f\sim \chi\cdot n^{it}\eeq
for some Dirichlet character $\chi$ and $t\in \R$ (the function $n\mapsto n^{it}$ is often called an Archimedean character).
Pretentious functions are thought of being not random,\footnote{However, even such functions can be of positive topological entropy, see e.g. \cite{Sa} for $f=\mu^2$.} but it is only the recent paper \cite{Fr-Le-Ru} which made this statement precise: all Furstenberg systems\footnote{Given $f\in\mathcal{M}$, by $(X_f,S)$, we denote the subshift generated by $f$: $X_f:=\ov{\{S^nf\colon n\in\Z\}}\subset \mathbb{U}^{\Z}$ ($\mathbb{U}$ stands for the unit disc), with $S$ being the left shift in $\mathbb{U}^{\Z}$; a Furstenberg system of $f$ is any measure-theoretic system $(X_f,\nu,S)$, where $\nu$ is an $S$-invariant (Borel) probability measure obtained as: $\lim_{k\to\infty}\frac1{N_k}\sum_{n\leq N_k}\delta_{S^nf}$, the limit taken in the weak$^\ast$-topology. Statistical properties of $f$ are studied via ergodic properties of Furstenberg systems. We recall that $(X_f,\nu,S)$ has {\em discrete spectrum} if $L^2(X_f,\nu)$ is spanned by the eigenfunctions of the unitary operator $U_S$: $g\mapsto g\circ S$. Discrete spectrum is the simplest instance of zero entropy systems.

If there is only one Furstenberg system $(X_f,\nu,S)$ for $f$, that is, when $\lim_{N\to\infty}\frac1N\sum_{n\leq N}\delta_{S^nf}=:\nu$ exists, then
$f$ is called {\em generic} for $\nu$.} of such functions are of zero entropy (in fact, they have discrete spectrum, see below), however, still a refined classification inside this class is expected. We will return to this problem later on.

Slightly despite its formal definition, non-pretentiousness turns out to be both strong and clear notion, namely, in view of Corollary~1 \cite{Da-De}, $f\in \cm$ is non-pretentious if and only if it is {\em aperiodic}, i.e.\ its Ces\`aro averages along every infinite arithmetic progression are tending to zero. While by the Chowla conjecture \cite{Ch}, \cite{Sa} the classical multiplicative functions like Liouville and M\"obius  are predicted of being extremely random (e.g., conjecturally, the Liouville function is a generic point for the Bernoulli measure $(1/2,1/2)$ for the full shift in $\{-1,1\}^{\Z}$), until 2015, such extremal randomness was also thought to hold for all other aperiodic functions from $\cm$ (Elliott's conjecture, see e.g.\ \cite{Ma-Ra-Ta}). In its slightly simplified form, Elliott's conjecture asserted that if $f_1,\ldots,f_k\in\mathcal{M}$ and for some $1\leq j_0\leq k$, $f_{j_0}$ is aperiodic then
\beq
\label{elliott1}
\lim_{N\to\infty}\frac1N\sum_{n<N}\prod_{j=1}^k f_j(n+b_j)=0\eeq
for any choice of integers $b_1<\ldots<b_k$ and all $k\geq1$.
However, by disproving Elliott's conjecture in \cite{Ma-Ra-Ta}, on one hand, the authors of \cite{Ma-Ra-Ta} introduced the concept of {\em strongly non-pretentious} (also called, {\em strongly aperiodic})\footnote{The idea is to expect a quantitative (with respect to $t$)  condition on non-pretentiousness: for all $Q\geq1$,
 $$\lim_{X\to\infty}\inf_{|t|\leq X}\min_{\chi~{\rm mod}~q, q\leq Q}\sum_{p\leq X}\frac1p\big(1-{\rm Re}(f(p)\chi(p)p^{it})\big)=+\infty.$$ The Liouville or M\"obius functions are strongly aperiodic.} functions for which a corrected Elliott's conjecture (cE-conjecture), in particular, the Chowla conjecture, is expected to hold, while on the other side, they showed that there are aperiodic functions for which the Chowla conjecture fails. The latter class of aperiodic functions was called the MRT class in \cite{Go-Le-Ru} and it was shown there that each MRT function displays a mixed behaviour: its Furstenberg systems vary from extremely random (the Chowla conjecture holds along a subsequence) while there is also a variety of zero entropy Furstenberg systems of algebraic origin (non-ergodic, unipotent systems of tori). In \cite{Fr-Le-Ru}, even more zero entropy (i.e.\ deterministic) Furstenberg systems have been discovered (not necessarily of algebraic origin) but the full classification  of Furstenberg systems (up to measure-theoretic isomorphism) of MRT functions still seems to be out of reach. To see some more complications, in \cite{Kl-Ma-Te}, the authors introduced two further subclasses of aperiodic multiplicative functions: those which are {\em moderately aperiodic}\footnote{\label{foot:MA} For some $A>0$, we have
$$\lim_{X\to\infty}\frac1{\log\log X}\inf_{|t|\leq X^A}\min_{\chi~{\rm mod}~q, q\leq (\log X)^A}\sum_{p\leq X}\frac1p\big(1-{\rm Re}(f(p)\chi(p)p^{it})\big)=0.$$
Note that this condition is satisfied for any pretentious function, that is why we assume $f$ to be aperiodic.} for which the Chowla conjecture holds along a subsequence and also the class of Liouville-like functions $f(n)=\lambda_{\mathcal{P}}(n):=(-1)^{\omega_{\mathcal{P}}(n)}$ ($\omega_{\mathcal{P}}(n)$ counts the number of prime divisors of $n$ belonging to $\mathcal{P}$), where $\mathcal{P}\subset\mathbb{P}$ has zero density in $\mathbb{P}$ and $\sum_{p\in\mathcal{P}}\frac1p=+\infty$ which do satisfy the Chowla conjecture. Some  relations between the aforementioned subclasses of aperiodic functions will be discussed in Section~\ref{s:relacje}.

\subsection{Main aim of the paper}
Our first task in this paper is to investigate various subclasses of pretentious functions. It can be seen as a continuation of investigations in \cite{Fr-Le-Ru}, that is, studying statistical properties of pretentious functions using dynamical tools. Our main aim is to prove that in the class of pretentious functions in $\mathcal{M}$ we have the following relations between the natural subclasses:

\beq\label{mapapret}
\begin{array}{c}
\text{ periodic } \subsetneq \text{  automatic } \subsetneq \text{  Toeplitz } \subsetneq \\ \text{ Besicovitch (rationally) almost periodic } = \\
\text{ unique Furstenberg system (generic)}\subsetneq \text{  pretentious}.\end{array}\eeq
Moreover, we provide purely number-theoretic characterizations of all above subclasses, given in Proposition~\ref{prop:periodic}, Proposition~\ref{prop:auto_Toeplitz}, Theorem~\ref{t:glowne} and  Theorem~\ref{t:OK}, respectively.  We pass now to a more precise description of the results together with some motivations, applications and all necessary definitions.

\subsubsection{Pretentious multiplicative functions, Furstenberg systems and Toeplitz sequences}For the pretentious functions, as shown in \cite{Fr-Le-Ru}, the ergodic components of all Furstenberg systems are odometers, i.e.\ algebraic systems being rotations on inverse limits of (finite) cyclic groups. More precisely, they are considered with the addition of~1 and Haar measure being the unique invariant measure and, as measure-theoretic dynamical systems, they are ergodic and have discrete spectrum (hence, they are of zero entropy), where the spectrum consists of roots of unity of degrees given by the order of the relevant cyclic groups. Moreover, if in~\eqref{pret1} $t=0$,  then all Furstenberg systems are ergodic and pairwise isomorphic (we can have however uncountably many of them). In dynamics, the simplest class of (topological) systems ``producing'' odometers are subshifts determined by the regular Toeplitz sequences.\footnote{\label{ft:reg} A sequence $f=(f(n))$ over an arbitrary ``alphabet'' is called Toeplitz if every position $n$ has its own period: there exists $k_n\geq1$ such that $f(n+sk_n)=f(n)$ for all $s\in\Z$. Then, in fact, there exists a period structure, that is, there exists a sequence $(m_n)$ of essential periods (see Remark \ref{rem:terminology}) $m_n|m_{n+1}$, $n\geq1$ \cite{Wi}. If $\ell_n$ denotes the number of positions $j\leq n$ whose period is $\leq m_n$, then $f$ is called {\em regular} provided that $\ell_n/m_n\to1$. If additionally the alphabet is compact then we can speak about the subshift $(X_f,S)$ generated by $f$. Then, for $f$ being a regular Toeplitz sequence the corresponding subshift is uniquely ergodic, and the associated odometer is given by the inverse limits of $\Z/m_n\Z$ (the numbers $e^{2\pi i/m_n}$ are eigenvalues of the unitary operator $U_S$ acting on $L^2(X_f,\nu)$).}\\
The following question arises naturally: which multiplicative sequences are Toeplitz? Note that if $f\in\mathcal{M}$ is Toeplitz then $f$ has to be pretentious, because along the arithmetic sequence determined by the period of the position~1 the sequence $(f(1+k_1s))_{s}$ is constant and equals to~1 (whence the mean along this sequence is non-zero). In fact,
\beq\label{pret2}\mbox{if $f\in\cm$ is Toeplitz then $f\sim \chi$}\eeq
for some Dirichlet character $\chi$. Indeed,
we will repeat the argument used in the survey \cite{Fe-Ku-Le} (see the proof of (15) in Section 2.3 therein). Suppose that
$\mathbb{D}(f,\chi\cdot n^{it})<+\infty$ for some Dirichlet character $\chi$ and $t\neq0$. Since $\mathbb{D}(f,\chi\cdot n^{it})=\mathbb{D}(f\cdot\ov{\chi}, n^{it})$ and clearly $f\cdot\ov{\chi}$ is also Toeplitz, we can assume that
$$
\mathbb{D}(f,n^{it})^2=\sum_{p}\frac1p\big(1-{\rm Re}(f(p)p^{it})\big)<+\infty.$$
Consider the arithmetic sequence $m\N_0+1$, where $m$ is a period for the position~1. Then $f$ is constant equal to~1 along that sequence, so in particular, we obtain $$\sum_{p\in m\N_0+1}\frac1p\big(1-\cos(t\log p)\big)<+\infty.$$
By the PNT in arithmetic progressions, there exists a constant $C>0$ such that for all $k\geq k_0$, we have
$$
\Big|\Big\{p\in m\N_0+1\colon e^{\frac{2\pi}t(k+\frac16)}\leq p\leq e^{\frac{2\pi}t(k+\frac56)}\Big\}\Big|\geq C\cdot \frac1{\phi(m)}\cdot\frac{e^{2\pi k/t}}{k/t}.$$
Equivalently,
$$\Big|\Big\{p\in m\N_0+1\colon
k+\frac16\leq \frac{t\log p}{2\pi}\leq k+\frac56\Big\}\Big|\geq
C\cdot \frac1{\phi(m)}\cdot\frac{e^{2\pi k/t}}{k/t}.$$
It follows that, for some constants  $D,E>0$ (depending on $t$, $m$),
$$
\sum_{e^{\frac{2\pi}t(k+\frac16)}\leq p\leq e^{\frac{2\pi}t(k+\frac56)}, p\in m\N_0+1}\frac1p\Big(1-\cos(t\log p)\Big)\geq$$$$ D\cdot\frac1{e^{\frac{2\pi}t(k+\frac56)}}\cdot\frac{e^{2\pi k/t}}{k/t}\geq E\cdot \frac1k$$
and the summation over $k\geq k_0$ yields a contradiction.

Note also that~\eqref{pret2} follows immediately from Theorem~\ref{t:glowne} (below). The reasoning above does not provide us a refined information on $\chi$. We will derive more precise information on $\chi$ in Remark~\ref{rem:pushout_nowa}.\footnote{We also consider the case of periodic multiplicative functions, see Proposition~\ref{prop:periodic} (periodic functions are both automatic and Toeplitz).}

\subsubsection{Automatic sequences and multiplicative functions} Yet, there are other motivations to study multiplicative functions which are Toeplitz coming from the theory of automatic sequences.\footnote{A sequence $(a_n)$ is called $r$-{\em automatic} if its $r$-{\em kernel}, i.e.\ the family of sequences
$$\{n\mapsto r^kn+j\colon k\geq1, 0\leq j<r^k\}$$
is finite. Automatic sequences are exactly those produced by finite automata, see e.g.\ \cite{Al-Sh}.}  In \cite{Be-Br-Co}, it was conjectured (Bell-Bruin-Coons' conjecture, for short, the BBC conjecture) that for each multiplicative automatic sequence $f$ there exists an eventually periodic $g$ such that $f(p)=g(p)$ for all prime numbers $p$. This conjecture was settled in \cite{Kl-Ku}, \cite{Ko}, and following \cite{Ko}, for each multiplicative automatic sequence there exists $\chi$ being a Dirichlet character or all the zero sequence such that
\beq\label{rowf}\mbox{$f(n)=\chi(n)$ for all $n$ coprime with a finite set $F\subset\mathbb{P}$.}\eeq A certain problem with the BBC conjecture is that it tells us about a property of multiplicative functions which is definitely weaker than automaticity. Indeed, first of all, the automaticity of a sequence implies that the used alphabet must be finite, while condition~\eqref{rowf} can be even satisfied for a sequence whose set of values is infinite.
More importantly, even for a finite alphabet, condition~\eqref{rowf} is weaker than automaticity.
Indeed, note that the condition in \eqref{rowf} is closed under multiplication:\footnote{Also the BBC conjecture itself is closed under multiplication.} if $f',\chi',F'$ also satisfy \eqref{rowf}, then
\beq\label{ilo}
f(n)f'(n)=\chi(n)\chi'(n)\text{ for all }n\text{ coprime with }F\cup F'\eeq
and $\chi\cdot \chi'$ is still a Dirichlet character. However, if $f$ is $p$-automatic, and $f'$ is $q$-automatic then (in general) $f\cdot f'$ is  $r$-automatic for no $r\in\mathbb{P}$. A relevant example can be found in Appendix, Section~\ref{s:przyklad}.

One of the main results of the paper says that  the (slightly modified) BBC-conjecture is nothing but a characterisation\footnote{\label{f:dziesiatka} Note that if \eqref{rowf} is satisfied then for each $n\geq 1$, $n=n'\prod_{p\in F}p^{a_p}$ with $n'$ coprime to $F$, for each $s$, we have ($d$ stands for the modulus of $\chi$)
$$f\Big(n+sd\prod_{p\in F}p^{a_p+1}\Big)=f\Big(\prod_{p\in F}p^{a_p}\Big(n'+sd\prod_{p\in F}p\Big)\Big)=$$$$
f\Big(\prod_{p\in F}p^{a_p}\Big)f\Big(n'+sd\prod_{p\in F}p\Big)=f\Big(\prod_{p\in F}p^{a_p}\Big)\chi\Big(n'+sd\prod_{p\in F}p\Big)=$$$$
f\Big(\prod_{p\in F}p^{a_p}\Big)\chi(n')=f(n),$$
so the sufficiency in Theorem~\ref{t:glowne} is obvious.
}
of Toeplitz multiplicative sequences:

\begin{Th}\label{t:glowne} Let $f:\N\to\C$ be a nonzero multiplicative function. Then $f$ is Toeplitz if and only if~\eqref{rowf} is satisfied, i.e.\
there are a Dirichlet character  $\chi$  and a finite set $F\subset\mathbb{P}$ such that $f(n)=\chi(n)$ for all $n$ coprime with all numbers in $F$. Moreover, each multiplicative Toeplitz sequence must be regular.
\end{Th}

\begin{Remark} \label{r:ml1} As bounded regular Toeplitz sequences determine uniquely ergodic systems, we have shown (in particular) that whenever $F\subset\mathbb{P}$ is finite and \eqref{rowf} holds then $f$ has exactly one Furstenberg system. It is natural to ask what happens if we take $F$ to be infinite and ``small'', say $\sum_{p\in F}\frac1p<+\infty$. For example, can we find an infinite $F\subset \mathbb{P}$, $\sum_{p\in F}\frac1p<+\infty$ satisfying $f(n)=\chi(n)$ for $n$ coprime with $F$  (note that these conditions imply $f\sim \chi$) such that the mean $M(f)$ of $f$ does not exist?\footnote{Such functions must have uncountably many Furstenberg systems for $f$. The first examples of functions pretending to be Dirichlet characters and having uncountably many
Furstenberg systems were given in \cite{Fr-Le-Ru}. All such examples are built over infinite alphabets. As Corollary~\ref{c:FV} shows, it is impossible to find such an example in the class of pretentious functions taking finitely many values.} It is slightly surprising that the answer to this question is negative. In fact, in Corollary~\ref{c:male1}, we will show that under the above assumptions, we are still in the class of multiplicative functions having exactly one Furstenberg system.
\end{Remark}

\subsubsection{Precise description of automatic sequences: periodic multiplicative functions and $p$-valuation}
Recalling that the BBC conjecture was not aimed at a characterization of automatic sequences which are multiplicative, in \cite{Ko-Le-Mu}, the problem of precise description of all automatic sequences which are multiplicative has been solved. If $(a_n)_{n\ge  1}$ is such a sequence then it must be of the form ($\nu_p(n)$ stands for the $p$-valuation of $n$):
\beq\label{form1}
a_n=f_1(\nu_p(n))f_2\Big(\frac{n}{p^{\nu_p(n)}}\Big)\eeq
for some $p\in\PP$, $f_1:\N_0\to\C$ eventually periodic, $f_1(0)=1$ and an eventually periodic $f_2:\N\to\C$ which is multiplicative such that $f_2(p^k)=0$ for $k\ge 1$.
We shall call an automatic multiplicative sequence $(a_n)$ {\em nonsingular} if it admits the presentation (\ref{form1}) with $f_2$ nonzero periodic.\footnote{\label{f:zeroBnorm} Note that in the singular case, \eqref{hi18} holds. It follows that there exists a finite set $S\subset\N$ such that $g_2(n):=f_2\Big(\frac{n}{p^{\nu_p(n)}}\Big)$ takes non-zero values only at numbers $mp^k$ with $m\in S$ and $k\geq1$. The set of such numbers has clearly zero density, and therefore $g_2$ is not pretentious (hence it is aperiodic). Note that by the same token, for any $g_1\in\mathcal{M}$, the function $n\mapsto g_1(n)g_2(n)$ is also aperiodic. It follows that if the automatic sequence \eqref{form1} is pretentious, it must be nonsingular.}
As each multiplicative function $f_2:\N\to\C$ which is eventually periodic either has finite support\footnote{\label{f:finitesupp} and then $f_2$ is fully described by the condition
\beq\label{hi18}|\{(q,k)\in\PP\times\N\colon f_2(q^k)\neq0\}
|<+\infty.\eeq}
or $f_2$ is multiplicative and periodic \cite{Ka0} (see also e.g.\ \cite{Ko-Le-Mu}, Lemma~3.2), the problem of classifying automatic multiplicative sequences is reduced to the classification of periodic multiplicative sequences. The latter is present in the old literature\footnote{Classically, periodic completely multiplicative functions are precisely Dirichlet characters.}, see \cite{Ka0}, \cite{Ka}, \cite{Gl}, \cite{Le-Wo}  but, for sake of completeness, in Appendix, Section~\ref{s:periodic}, we give a separate proof\footnote{Let us add that the proof of Theorem 3 \cite{Ka-Pe} also seems to give some characterization of periodic multiplicative functions (and the methods applied there can be pushed forward to get the detailed description). Also, the proof of Proposition~4.4 in \cite{Kl} yields a description whenever the value at the period is non-zero.}
using Theorem~\ref{t:glowne}.

Next proposition gives a complete description\footnote{Note that condition~\eqref{eq:newformula} reads: for all $\ell\geq0$,\\
(i) if $q|t$ then $f(q^{(\nu_q(M)-\nu_q(t))+\ell})=f(q^{\nu_q(M)-\nu_q(t)})\theta(q^{\ell})$,\\
(ii) if $q|M$ and ${\rm gcd}(q,t)=1$ then $f(q^{\nu_q(M)+\ell})=f(q^{\nu_q(M)})\theta(q^{\ell})$,\\
(iii) if ${\rm gcd}(q,M)=1$ then $f(q^{\ell})=\theta(q^\ell)$.

In Proposition~\ref{p:odwrotny}, we will show how an abstract version of (i)-(iii) defines a periodic multiplicative function.}
of periodic multiplicative sequences:


\begin{Prop}\label{prop:periodic} Let $f$ be a nonzero multiplicative periodic function of period $M$. There exists a Dirichlet character $\theta$ modulo $t$, where $t|M$,  such that $f(n)=\theta(n)$ for all $n$ coprime with $M$. Moreover,
\begin{equation}\label{eq:newformula}
f(q^{\nu_q(M/t)+\ell})=f(q^{\nu_q(M/t)})\theta(q^{\ell})
\end{equation}
for every prime $q$ and $\ell\ge 0$. In particular, $f(q^{\nu_q(M/t)+\ell})=0$ for $q|t$ and $\ell>0$.
\end{Prop}

Furthermore, we characterize multiplicative automatic sequences as Toeplitz sequences with  a simplest period structure (for the definitions of $Per(f)$ and $\nu_p(Per(f))$, see Remark~\ref{rem:terminology}):

\begin{Prop}\label{prop:auto_Toeplitz}
Assume that $f$ is a nonzero multiplicative function. The following conditions are equivalent:
\begin{enumerate}
\item[(a)] $f$ is automatic nonsingular.
\item[(b)] $f$ is either periodic or nonperiodic Toeplitz and in the latter case there exists a unique prime $p$ with $\nu_p(Per(f))=\infty$ and the sequence $(f(p^k))_k$ is eventually periodic for this $p$.\footnote{It can be shown that, for an arbitrary Toeplitz multiplicative function $f$, if $\nu_q(Per(f))<+\infty$ for a prime $q$, then the sequence $(f(q^k))_k$ is eventually periodic. This is an easy consequence of the formula (\ref{eq:newformula}) in Proposition~\ref{prop:periodic} in the case of periodic $f$. For the remaining $f$, one first proves that the function $f'=f\cdot\chi$ is periodic, where $\chi$ is the principal Dirichlet character of modulus $\prod_{p:\nu_p(Per(f))=+\infty}p$  (here the observation in Footnote \ref{f:23} helps) and then applies the above observation to $f'$.}

\end{enumerate}
\end{Prop}


\subsubsection{Pretentious functions with one Furstenberg system. The Frantzikinakis-Host conjecture} We consider now only functions from $\mathcal{M}$. As Theorem~\ref{t:glowne} shows, nonzero Toeplitz multiplicative functions (which are necessarily pretentious)  have exactly one Furstenberg system. In general, the classification of Furstenberg systems of pretentious functions has been accomplished in \cite{Fr-Le-Ru}: first of all Furstenberg systems of a fixed pretentious function are isomorphic (and either there is a one Furstenberg system or we have uncountably many of them). In case of $f\sim\chi$ any Furstenberg system is ergodic (more precisely, it is an ergodic odometer), and if $f\sim\chi\cdot n^{it}$, $t\neq 0$, any Furstenberg system is not ergodic, it has uncountably many ergodic components, each component being an ergodic odometer. It is not clear however from \cite{Fr-Le-Ru} how to recognize from $f$ that it has exactly one Furstenberg system. Using ideas from \cite{Da-De}, \cite{Kl}, \cite{Be-Ku-Le-Ri0}, we will show the following:

\begin{Th}\label{t:OK}
Let $f\in \mathcal{M}$. Suppose that $f$ is pretentious, $\mathbb{D}(f,\chi\cdot n^{it})<+\infty$ for some primitive Dirichlet character $\chi$ and $t\in\mathbb{R}$. Then the following conditions are equivalent:\\
(i) $f$ has a unique Furstenberg system, i.e.\ $f$ is generic. \\
(ii) $t=0$ and the series \beq\label{mlok}
\sum_p\frac1p\Big(1-f(p)\ov{\chi(p)}\Big)\text{ converges}.\footnote{Note that this condition is stronger than $\mathbb{D}(f,\chi)<+\infty.$}\eeq
(iii) $f$ is RAP.\footnote{\label{f:rap} $f$ is called Besicovitch rationally almost periodic (RAP, for short) if for each $\vep>0$ there exists a periodic sequence $f_{\vep}$ such that $\limsup_{N\to\infty}\frac1N\sum_{n<N}|f(n)-f_{\vep}(n)|<\vep$. Of course, each regular Toeplitz sequence is RAP.}

Moreover, the unique Furstenberg system is an ergodic odometer.\end{Th}

As a consequence of this theorem, we prove that each real-valued pretentious $f\in\mathcal{M}$ has a unique Furstenberg system (which is an ergodic odometer). The same also holds for all pretentious functions taking finitely many values.

In \cite{Fr-Ho}, Frantzikinakis and Host formulated Conjecture~1  which asserts that each real-valued $f\in\mathcal{M}$ has exactly one Furstenberg system (we call it the FH-conjecture). Hence, we verified positively this conjecture in the pretentious case and we will additionally show that the following holds.

\begin{Cor}\label{c:FH10} The cE-conjecture implies  the FH-conjecture.
\end{Cor}

\subsection{The GHK semi-norms of pretentious functions}
In Section~\ref{s:ghk}, we discuss the problem of Gowers-Host-Kra semi-norms $\|f\|_{u^s}$, $s\in\N$, of pretentious functions $f$ in $\mathcal{M}$. While the results here are surely known for aficionados, we provide proofs for completeness. Due to \cite{Fr-Le-Ru}, we provide a dynamical proof of the fact that all the semi-norms are positive for $s\geq2$. Concentrating on $s=1$, we show that if $f\in\mathcal{M}$ is a pretentious function then $\|f\|_{u^1}=0$ if and only if the mean of $f$ is equal to~0.

Recall that $f\in\mathcal{M}$ is said to satisfy the {\em local 1-Fourier uniformity}, if
\beq\label{l1fu1}
\lim_{H\to\infty}\limsup_{M\to\infty}
\frac1M\sum_{m<M}\Big\|\frac1H\sum_{h<H}f(m+h)e^{2\pi ih\cdot}\Big\|_{C(\T)}=0\eeq
In Section~\ref{s:ghk}, we also provide a proof of the following fact:
\begin{Prop}\label{p:nonl1fu}
No pretentious $f\in\mathcal{M}$ satisfies the local 1-Fourier uniformity property.\end{Prop}

The problem whether local 1-Fourier uniformity holds for strongly aperiodic functions is widely open. On the other hand, the MRT functions fail to satisfy this property \cite{Go-Le-Ru}.

\subsection{Relations between subclasses of aperiodic functions} While \eqref{mapapret} gives a rather complete map of natural subclasses (and relations between them) of pretentious functions from $\mathcal{M}$, the situation is much less clear for the class of aperiodic functions. In Section~\ref{s:relacje}, we show that:
\beq\label{mapaaper}\begin{array}{c}
\mbox{[strongly aperiodic $\cup$ moderately aperiodic] $\subsetneq$ aperiodic},\\
\mbox{Liouville-like $\subset$ [strongly aperiodic $\cap$ moderately aperiodic]}.\end{array}\eeq
We also show that there is a strongly aperiodic function in $\mathcal{M}$ which is not moderately aperiodic.  It is an interesting question whether there is a moderately aperiodic function which is not strongly aperiodic.

Moreover, by \cite{Ma-Ra-Ta},
$$\mbox{MRT} \cap \mbox{strongly aperiodic} =\emptyset,$$ while possible relations between the MRT class and moderately aperiodic functions seem to constitute an open problem.

\section{Toeplitz case -- proof of Theorem~\ref{t:glowne}}\label{s:toeplitz}
\begin{Remark}\label{rem:terminology}
Given a sequence $f=(f(n))\in\mathbb{C}^{\N}$, set
$$Per_m(f):=\{n\in\N\colon f(n)=f(n')\text{ for all }n=n'\text{ mod }m\}.$$
Then $f$ is called Toeplitz if $\N=\bigcup_{m\geq 1}Per_m(f)$. For each $m$, by the $m$-{\em skeleton} of $f$, one means the $m$-periodic sequence $f^{(m)}$ arising from $f$ by replacing the symbols at all positions from $\N\setminus Per_m(f)$ with a ``hole'', i.e.\ a new symbol. Whenever for $f^{(m)}$ the number $m$ is the smallest period, we say that $m$ is an {\em essential period} of $f$.
Note that if $m'$ is another essential period of $f$ so also is $\lcm(m,m')$. Indeed, if not then there exists $d|{\rm lcm}(m,m')$ which is a period of the ${\rm lcm}(m,m')$-skeleton $f^{({\rm lcm}(m,m'))}$ of $f$. If in this skeleton we replace the symbols at all positions from $Per_{{\rm lcm}(m,m')}(f)\setminus Per_m(f)$ with the holes, we still have a $d$-periodic sequence. But this sequence is exactly $m$-skeleton of $f$, so $m|d$, and our claim now follows.
\newline
By listing  $r_1<r_2<\ldots $ all essential periods and taking $m_n:={\rm lcm}(r_1\ldots r_n)$, we obtain so called {\em period structure} of $f$, that is,  a sequence $m_1|m_2|\ldots$ of essential periods such that $\N=\bigcup_k Per_{m_k}(f)$, see \cite{Wi}.  For a period structure $\underline{m}=(m_1|m_2|\ldots)$ of $f$ and a prime $p$ we denote by $\nu_p(\underline{m})$ the {\em $p$-valuation of  $\underline{m}$}, that is,  the largest number $b\in\N_0$ such that $p^b$ divides some $m_i$, if such a number $b$ exists; otherwise, we set $\nu_p(\underline{m})=\infty$. It is well known that  $\nu_p(\underline{m})$ does not depend on the choice of the period structure.\footnote{Assuming additionally that $f$ is bounded, a dynamical proof of this fact consists in showing that the associated $\underline{m}$-odometer is the maximal equicontinuous factor of the subshift determined by $f$ \cite{Wi}. Thus, if $\underline{m}'$ is another period structure, for each $i\geq1$, there is $j_i$ such that $m_i|m'_{j_i}$  (and vice-versa).}   Here is a simple argument: assume that $\underline{m}=(m_1|m_2|\ldots)$ and $\underline{m'}=(m'_1|m'_2|\ldots)$ are two period structures of $f$ and $\nu_p(\underline{m})=:b<\infty$. It follows that for every position there is an associated period $m$ with $\nu_p(m)\leq b$. Assume that $\nu_p(m'_i)> b$ for some $i$. Then, Lemma~\ref{lem:general} below yields $Per_{m'_i}(f)=Per_{p^{b-\nu_p(m'_i)}m'_i}(f)$, contradicting the fact that $m'_i$, being a member of a period structure, is an essential period. Therefore, we call $\nu_p(\underline{m})$ the {\em $p$-valuation of the periods of $f$} and denote it by $\nu_p(Per(f))$.
Note that the above argument yields that
\beq\label{uw:mul}
\begin{array}{c}\mbox{if for every position there is an associated period $m$ with  $\nu_p(m)\leq b$}\\
\mbox{then $\nu_p(Per(f))\le b$.}\end{array}\eeq
We shall use this observation several times, in particular, in the proof of Proposition~\ref{prop:auto_Toeplitz}.
By the {\em spectrum of a period structure} of $f$, denoted by $\spec(Per(f))$, we mean the set of the primes $p$ such that $\nu_p(Per(f))>0$.
\end{Remark}

\begin{Lemma}\label{lem:general}
Let $f$ be Toeplitz (not necessarily multiplicative). Let $p$ be a fixed prime and $b\in\N_0$. Assume that for every $s\in\N$ there exists a number $M\in\N$ such that $s\in Per_M(f)$ and $\nu_p(M)\leq b$.
Let $n,K\in \N$  and $n\in Per_K(f)$. Then $n\in Per_{\overline{K}}(f)$, where $\overline{K}=p^{b-\nu_p(K)}K$.
\end{Lemma}
\begin{proof}
Take arbitrary $k\in\N_0$ and let $M$ be a period associated to the position $n+\overline{K}k$ and such that $\nu_p(M)\leq b$. Then $\gcd(M,K)|\overline{K}$, so there exist $x,y\in\N$ such that
$$
xK-yM=\overline{K}.
$$
It follows that
$$
n+xKk=n+\overline{K}k+yMk,
$$
whence
$$
f(n)=f(n+xKk)=f(n+\overline{K}k+yMk)=f(n+\overline{K}k),
$$ since $n\in Per_K(f)$ and $n+\overline{K}k\in Per_M(f)$, so $n\in Per_{\overline{K}}(f)$.
\end{proof}

Below, we consider $f$ which is Toeplitz and multiplicative (no restriction on the boundedness of $f$).

\begin{Remark}\label{r:uw2}
Assume that $\theta$ is a Dirichlet character of modulus  $m$. Let $q_1,\ldots,q_u$ be distinct primes. We modify $\theta$ at $q_1,\ldots,q_u$: let $\kappa_j:\N\rightarrow\C$, $j=1,\ldots,u$ be arbitrary functions. We define the multiplicative function $f$ by setting $f(1)=1$ and:
$$
f(p^k)=\left\{\begin{array}{l} \theta(p^k)\;\text{if}\;p\notin\{q_1,\ldots,q_u\},\\
\kappa_j(k)\;\text{if}\;p=q_j,j=1,\ldots, u,
\end{array}\right.
$$
where $p$ is a prime and $k\geq 1$.
\end{Remark}

\begin{Lemma}\label{lem:aboutexample} We keep the notation from Remark \ref{r:uw2}, in particular, let $f$ be the function defined there.
\begin{enumerate}
\item[(a)] Let $n\in\N$ and write $n=q_1^{v_1}\ldots q_{u}^{v_{u}}n'$, where $v_j=\nu_{q_j}(n)$ is the $q_j$-valuation of $n$. Then $n\in Per_T(f)$,
    where $T:=q_1^{v_1+1}\ldots q_{u}^{v_{u}+1}m$.
\item[(b)] $f$ is Toeplitz and  $\spec(Per(f))\subseteq \spec(m)\cup\{q_1,\ldots,q_u\}$.
\item[(c)] If in addition  every function $\kappa_j$ is of the form $\kappa_j(k)=\xi_j^k$ for some $\xi_j\in\C$, then $f$ is completely multiplicative.
\end{enumerate}
\end{Lemma}
\begin{proof}
(a) (cf.\ Footnote~\ref{f:dziesiatka}) Obviously, the number $n'$ is coprime with $q_1,\ldots,q_u$, so we have
\begin{equation}\label{eq:valuen}
f(n)=f(q_1^{v_1}\ldots q_u^{v_u})\theta(n').
\end{equation}
 For each $k\geq1$, the number $n'+kq_1\ldots q_um$ is coprime with $q_1,\ldots,q_u$, therefore, by the multiplicativity of $f$,  $m$-periodicity of $\theta$ and~\eqref{eq:valuen},
$$
f(n+kT)=f(q_1^{v_1}\ldots q_u^{v_u})f(n'+kq_1\ldots q_um)=
$$$$
f(q_1^{v_1}\ldots q_u^{v_u})\theta(n'+kq_1\ldots q_um)=$$$$f(q_1^{v_1}\ldots q_u^{v_u})\theta(n')=f(q_1^{v_1}\ldots q_u^{v_u})f(n')=f(n).$$

(b) follows from (a) via \eqref{uw:mul}, whereas
(c) is obvious.
\end{proof}

Below, we prove in Theorem~\ref{th:chiporaz1} that every nonzero  Toeplitz multiplicative function has the form as described in Remark~\ref{r:uw2}.

\begin{Example}\label{ex:mod3} Let $\chi$ be the Dirichlet character of modulus~3 given by $\chi(2)=-1$. Consider
$$
f(n):=\chi(n/2^{\nu_2(n)}),$$
i.e., for all $k\geq1$, $f(2^k)=1$ and $f(p^k)=\chi(p^k)$ for the remaining primes.
It follows that $f$ is a completely multiplicative function which is Toeplitz. We will show that $f$ is not periodic. More precisely, we show that $\nu_2(Per(f))=\infty$. Suppose otherwise, and denote $b=\nu_2(Per(f))<\infty$. By Lemma~\ref{lem:aboutexample} (a), if $n=2^bn'$  (with $n'$ odd) then $2^{b+1}\cdot 3$ is a period for the position $n$. Assume $n=2^b$, all we need to show is that $2^{b+1}\cdot 3$ is the {\bf smallest} period for the position $2^b$. If this is not the case, then by    Lemma~\ref{lem:general},  $2^b\cdot3$ is a period associated to the position $2^b$:
\beq\label{jesien2}
f(2^b+2^{b}\cdot 3s)=f(2^{b})=\chi(1)=1\eeq
for all $s\geq1$. We have $$f(2^b+2^{b}\cdot 3s)=f(2^b(1+3s))=
f(2^b)f(1+3s)=f(1+3s).$$ We check however that for $s=3$, we have
$f(10)=\chi(10/2)=\chi(5)=-1$, a contradiction.
\end{Example}

Throughout, given $m\in\N$, we denote $m^{\perp}=\{n\in\N:\gcd(n,m)=1\}$.

\begin{Lemma}\label{eq:periodof1} Assume that $f$ is nonzero multiplicative and Toeplitz, and $m$ is a period associated to the position 1, that is, $1\in Per_m(f)$.  Then $m^{\perp}\subseteq Per_m(f)$ and $f(n)\neq 0$ for every $n\in m^{\perp}$.  Moreover, if $f$ is completely multiplicative, then $f(n)$ is a $\phi(m)$-root of 1.
\end{Lemma}
\begin{proof}
Let $\gcd(n,m)=1$. Then there exists $k\geq1$ such that $nk=1\mod m$, whence $f(nk)=f(1)=1$. Fix a number $x\in \N_0$. As $k$ is coprime with $m$, thanks to Dirichlet's theorem applied to the progression $k+m\Z$, there exists $y\in\N_0$ such that $k+ym$ is coprime with $n$ and also with $n+xm$.
Then
$$
f(n)f(k+ym)=f(nk+nym)=f(1)=1
$$
and
$$
f(n+xm)f(k+ym)=f(nk+(ny+xk+xym)m)=1.
$$
It follows that $f(n)=f(n+xm)\neq 0$. To prove the remaining statement note that $n^{\phi(m)}=1\mod m$, so
$1=f(1)=f(n^{\phi(m)})=f(n)^{\phi(m)}$, when $f$ is completely multiplicative.
\end{proof}

\begin{Lemma}\label{lem:completemult}
Assume that $f$  is  multiplicative and Toeplitz. Assume that $n,k\in\N$ and $n,nk\in Per_m(f)$ for a number $m$ such that  $\gcd(n,k,m)=1$. Then
$f(nk)=f(n)f(k)$.
\end{Lemma}
\begin{proof}
There exists $x\in\N_0$ such that $\gcd(n+xm,k)=1$. Indeed, by Dirichlet's theorem applied to $n/\gcd(n,m)+(m/\gcd(n,m))\N$, there exists $x\in\N_0$ such that $n+xm=\gcd(n,m)q$, where $q$ is a prime not dividing $k$. Then
$\gcd(n+xm,k)=\gcd(\gcd(n,m),k)=1$. It follows that
$$
f(nk)=f(nk+xkm)=f((n+xm)k)=f(n+xm)f(k)=f(n)f(k).
$$
\end{proof}

\begin{Lemma}\label{lem:zamiast1.9}
Assume that $f$ is nonzero multiplicative and Toeplitz and let  $m$ be a period associated to the position 1. Let $n'\in m^{\perp}$ and let $t\in\N$ be a number such that $\spec(t)\subseteq spec(m)$.
Then $tn'\in Per_{tm}(f)$. In particular, each position has a period with the spectrum contained in ${\rm \spec}(m)$.
\end{Lemma}
\begin{proof}
Let $x\in\N_0$. Then
$$
f(tn'+xtm)=f(t(n'+xm))=f(t)f(n'+xm)=f(t)f(n')=f(tn').
$$
Here, we use the fact that, since every prime factor of $t$ divides $m$ and $\gcd(n',m)=1$,  $\gcd(t,n'+xm)=\gcd(t,n')=1$, hence the middle term equality follows. We apply Lemma~\ref{eq:periodof1} (to $n'$) to justify the third equality. The last statement follows easily.
\end{proof}

\begin{Lemma}\label{lem:finitespec}
Assume that $f$ is nonzero multiplicative and Toeplitz and let  $m$ be a period associated to the position 1. Then $\spec(Per(f))\subseteq \spec(m)$, in particular, the spectrum of a period structure of $f$ is finite.
\end{Lemma}
\begin{proof}
It follows from Lemma~\ref{lem:zamiast1.9} that
 each position has a period with spectrum contained in ${\rm \spec}(m)$. Assume that $(m_n)$ is a period structure of $f$. If $q$ is a prime that is not in ${\rm \spec}(m)$ then, by   Lemma~\ref{lem:general} (applied for $p=q$ and $b=0$), the $m_n$-skeleton of $f$ coincides with the  $m_n/q^{\nu_q(m_n)}$-skeleton (see Remark \ref{rem:terminology}) for every $n$, which shows, as $m_n$ are essential periods, that $q\nmid m_n$. This shows that $\spec(Per(f))\subseteq \spec(m)$.
\end{proof}

\begin{Lemma}\label{lem:minperiod1}
Assume that $f$ is nonzero multiplicative and Toeplitz. Then
$$
\{m\in\N:1\in Per_{m}(f)\}=m_1\N,
$$
for some $m_1\in\N$, that is, the  minimal period $m_1$ associated to the position 1 divides  every period of 1. Moreover,  $\spec(Per(f))= \spec(m_1)$.
\end{Lemma}
\begin{proof}
By Lemma~\ref{lem:finitespec}, if $1\in Per_{m}(f)$ then $\spec(Per(f))\subseteq \spec(m)$. If $q\notin \spec(Per(f))$ for a prime $q$, then by Lemma \ref{lem:general} (applied to $p=q$, $b=0$, $n=1$ and $K=m$),
$1\in Per_{m/q^{\nu_q(m)}}(f)$, so every period $m$ associated to the position 1 is divisible by a period $m'$ associated to 1 such that $\spec(m')= \spec(Per(f))$. It follows that $\spec(Per(f))= \spec(m_1)$ if $m_1$ is the minimal period associated to the position 1. It remains to prove that if $t_1$ and $t_2$ are periods associated to 1, then $\gcd(t_1,t_2)$ is also such a period.  Let $t_1,t_2$ be numbers such that $1\in Per_{t_i}(f)$ and we can assume that additionally $\spec(t_i)=\spec(Per(f))$ for $i=1,2$.  Let $t=\gcd(t_1,t_2)$, thus $t=x_1t_1-x_2t_2$ for some $x_1,x_2\in\N$. Then, for every $k\in\N_0$, we have
$$
f(1+kt)=f(1+kx_1t_1-kx_2t_2)=f(1+kx_1t_1)=f(1).
$$
Indeed, the middle term equality follows since $\gcd(1+kx_1t_1-kx_2t_2,t_2)=1$ by the assumption that $\spec(t_1)=\spec(t_2)$, so $1+kx_1t_1-kx_2t_2\in Per_{t_2}(f)$ by Lemma~\ref{eq:periodof1}. It follows that $1\in Per_t(f)$.
\end{proof}


\begin{Prop}\label{cor:regular}
If $f$ is Toeplitz and multiplicative, then it is regular.
\end{Prop}

\begin{proof} It is enough to prove the statement for nonzero $f$.
By Lemma \ref{lem:minperiod1}, we know that the spectrum of $Per(f)$ is finite and $\spec(Per(f))= \spec(m_1)$, where $m_1$ is a period associated to the position 1. Let $\spec(Per(f))=\{q_1,\ldots,q_u\}$.
Let $n\in\N$ and write $n=q_1^{v_1}\ldots q_u^{v_u}n'$, where $n'$ is coprime with $q_1\ldots q_u$.  By Lemma~\ref{eq:periodof1}, $n'\in Per_{m_1}(f)$, so by  Lemma~\ref{lem:zamiast1.9} (applied to $n'$ and $t=q_1^{v_1}\ldots q_u^{v_u}$),  we deduce that
$$
n\in Per_{q_1^{v_1}\ldots q_u^{v_u}m_1}(f)
$$
for every $n\in\N$. It follows that
if
$T=q_1^{v_1}\ldots q_{u}^{v_{u}}m_1$
for some  $v_1,\ldots,v_u\in\N$, then $\N\setminus Per_T(f)$ is contained in the set of the multiples of the set $\{q_1^{v_1+1},\ldots, q_u^{v_u+1}\}$. The natural density of that set of multiples, hence also the density of the set $\N\setminus Per_{T}(f)$, converges to 0 as $v_j\rightarrow +\infty$ for $j=1,\ldots,u$. This proves the regularity of $f$.
\footnote{If $f$ is Toeplitz and $(t_i)_i$ is any sequence of natural numbers such that the natural density of the set $\N\setminus Per_{t_i}(f)$ converges to 0 as $i\mapsto +\infty$ then $f$ is regular. We do not need to assume that $t_i$ form a period structure.}
\end{proof}

\begin{Th}\label{th:chiporaz1}
Assume that $f$ is a nonzero Toeplitz multiplicative function and $1\in Per_m(f)$, where $\spec(m)=\spec(Per(f))$.\footnote{The existence of $m$ follows from Lemma \ref{lem:minperiod1}.}   There exists a unique  Dirichlet character $\chi$ of modulus   $m$  such that
$f(n)=\chi(n)$ for $n\in m^{\perp}$.~\footnote{\label{f:observe}Observe that $\chi$ does not depend on the choice of $m$, i.e. if $1\in Per_{m'}(f)$, where $\spec(m')=\spec(Per(f))$ and $\chi'$ is the unique Dirichlet character of modulus $m'$ satisfying $f(n)=\chi'(n)$ for $n\in {m'}^\perp$, then $\chi=\chi'$.}
\end{Th}

\begin{proof}   Let $\chi$ be the function defined by
$$
\chi(n)=\left\{\begin{array}{l}
f(n)\;\text{if}\; n\in m^{\perp},\\
0\; \text{otherwise}.
\end{array}\right.
$$
We claim that $\chi$ is $m$-periodic. Indeed, clearly, $n$ is coprime with $m$ if and only if  $n+m$ is also coprime with $m$, so   $\chi(n)=f(n)=f(n+m)=\chi(n+m)$ for $n\in m^{\perp}$. Besides, $\chi(n)=f(n)\neq0$ for $n\in m^\perp$ by Lemma~\ref{eq:periodof1}. If $n\notin m^{\perp}$ then
$\chi(n)=0=\chi(n+m)$ and the claim follows.

Moreover, $\chi$ is completely multiplicative. Indeed, if $n$ and $k$ are coprime with $m$, then by Lemma~\ref{eq:periodof1}, $n,nk\in Per_m(f)$ and by Lemma~\ref{lem:completemult}, $\chi(nk)=f(nk)=f(n)f(k)=\chi(n)\chi(k)$. Otherwise, $\chi(nk)=0=\chi(n)\chi(k)$. Therefore, $\chi$ is  $m$-periodic. It follows that $\chi$ is a  Dirichlet character of modulus $m$. Moreover, $f(n)=\chi(n)$ for $n\in m^{\perp}$ and this condition determines $\chi$ uniquely.
\end{proof}

\begin{Remark}\label{rem:pushout_nowa} Let us collect some facts on Dirichlet characters (comp.   \cite[Chapter 5]{Dav}).
\begin{enumerate}
\item[(a)]
Given a character $\chi$ there is the minimal modulus of $\chi$ which divides every other modulus of $\chi$. If $\chi$ is a Dirichlet character of modulus $m$ and $\chi$ is induced by a character $\theta$, that is, $\chi(n)=\theta(n)$ for every $n\in m^{\perp}$, then the minimal modulus of $\theta$ divides $m$.
\item[(b)] The relation of the induction is a partial order in the set of all Dirichlet characters, we write $\theta\preceq\chi$ if $\theta$ induces $\chi$. It turns out that the set of the characters $\psi$ such that $\psi\preceq\chi$ for a fixed $\chi$ has the smallest element $\theta$. This is the unique primitive (that is, not induced by any other character) character that induces $\chi$. The minimal modulus $t$ of $\theta$ is the {\em conductor}  of $\chi$\footnote{Differently than in \cite{Dav}, we allow that $\chi$ is a principal character in which case  $t=1$.}, see e.g. \cite[Chapter 4]{Koch}.
    For the convenience of the reader we sketch the main arguments.
The crucial observation is that, given a character $\chi$, for  every two characters $\theta',\theta''\preceq \chi$ there exists $\theta\preceq \theta',\theta''$.
Suppose that $\theta'$ and $\theta''$ are characters of moduli $m'$ and $m''$ that induce $\chi$. For $n\in\lcm(m',m'')^{\perp}$ we can find $k\in\N_0$ such that $n_1:=n+k\lcm(m',m'')\in m^{\perp}$ (we apply the usual argument based on  Dirichlet's theorem). Then $\theta'(n)=\theta'(n_1)=\theta''(n_1)=\theta''(n)$, so $\theta'$ and $\theta''$ coincide on $\lcm(m',m'')^{\perp}$. Now, take $n\in \gcd(m',m'')^{\perp}$. The set $(n+\gcd(m',m'')\N_0)\cap \lcm(m',m'')^{\perp}$ is nonempty and both $\theta'$ and $\theta''$ are constant on this set. Indeed, choose $x,y\in\N$ such that $xm'-ym''=\gcd(m',m'')$ and assume that $n_1,n_2:=n_1+k\gcd(m',m'')\in \lcm(m',m'')^{\perp}$ for some $k\in\N_0$. Then
$n_2=n_1+kxm'-kym''$ and it is easy to see that $n_3:=n_1+kxm'\in \lcm(m',m'')^{\perp}$. Therefore
$$\theta''(n_1)=\theta'(n_1)=\theta'(n_3)=\theta''(n_3)=\theta''(n_2)=\theta'(n_2),$$
as $\theta'$ and $\theta''$ coincide on $\lcm(m',m'')^{\perp}$ and $\theta'$ (resp. $\theta''$) is $m'$ (resp. $m''$) periodic. We define $\theta(n)$ to be the value of $\theta'$ and $\theta''$ on $(n+\gcd(m',m'')\N_0)\cap \lcm(m',m'')^{\perp}$, provided $n\in\gcd(m',m'')^{\perp}$ and $\theta(n)=0$ otherwise. One can prove that $\theta$ is a Dirichlet character that induces both $\theta'$ and $\theta''F$, so if the latter are primitive, that is: $\prec$-minimal, then $\theta'=\theta=\theta''$.
\item[(c)]
Now, let us turn to the situation of Theorem \ref{th:chiporaz1}. If $\chi$ is the (see Footnote~\ref{f:observe}) character corresponding to a Toeplitz multiplicative function $f$, then we call the unique primitive character $\theta$ that induces $\chi$ the {\em supporting character} of $f$. The conductor $t$ of $\theta$ (and $\chi$) divides every period of $f$ associated to the position~1 by Theorem \ref{th:chiporaz1}. Moreover, $t$ divides a modulus of every character that induces $\chi$. In particular, if $f$ is $M$-periodic for some $M\in\N$, then $t|M.$
\end{enumerate}
\end{Remark}

\noindent
{\em Proof of Theorem~\ref{t:glowne}.} The proof follows immediately  from Theorem~\ref{th:chiporaz1} and Proposition~\ref{cor:regular}.

\vspace{2ex}

\noindent
{\em Proof of Proposition~\ref{prop:auto_Toeplitz}.}
(a)$\Rightarrow$(b)  By Theorem 1.1 in \cite{Ko-Le-Mu} (see \eqref{form1}) there exist a prime $p$ and two functions $f_1$, $f_2$ such that for every $n\in\N$,
$$
f(n)=f_1(\nu_p(n))f_2\Big(\frac{n}{p^{\nu_p(n)}}\Big)
$$
and $f_1(0)=1$, $f_1$ is eventually periodic and, since $f$ in nonsingular,  $f_2$ is a nonzero multiplicative periodic function such that $f_2(p^k)=0$ for every $k>0$. Observe that both functions:
$$
n\mapsto  f_1(\nu_p(n))\;\text{and}\;n\mapsto  f_2\Big(\frac{n}{p^{\nu_p(n)}}\Big)
$$
are Toeplitz and for every prime $q\neq p$ the $q$-valuations of the periods of both of them are finite.
Indeed, this follows from Remark~\ref{rem:terminology} since for every $n,k\in\N$, we have:
$$
\nu_p(n+kp^{\nu_p(n)+1})=\nu_p(n)\;\text{and}\; f_2\Big(\frac{n+kp^{\nu_p(n)+1}m}{p^{\nu_p(n+kp^{\nu_p(n)+1}m)}}\Big)=f_2\Big(\frac{n}{p^{\nu_p(n)}}\Big),
$$ where $m$ is a period of $f_2$: use~\eqref{uw:mul} for $q$ (instead of $p$) and $b=0$ for the first function and $b=\nu_q(m)$ for the second function.
Moreover, $f(p^k)=f_1(k)f_2(1)=f_1(k)$\footnote{$f_2(1)=1$ since $f$ is nonzero.}, so the sequence $(f(p^k))_k$ is eventually periodic by the assumption on $f_1$. For every prime $q\neq p$ the $q$-valuation of the periods of $f$ is finite \footnote{\label{f:23} We apply the following observation: assume that $g_1$ and $g_2$ are Toeplitz functions (not necessarily multiplicative). Assume that $n\in Per_{k_1}(g_1)\cap Per_{k_2}(g_2)$. Then $n\in Per_{\lcm(k_1,k_2)}(g_1g_2)$. It follows  that, for every prime $q$, $\nu_q(Per(g_1g_2))\le \max(\nu_q(Per(g_1)),\nu_q(Per(g_2)))$.} and (b) follows.

(b)$\Rightarrow$(a) The statement is clear when $f$ is periodic. Suppose otherwise and let $p$ be the unique prime with infinite valuation of the periods of  $f$. We define $f_1(k):=f(p^k)$. As $f$  is nonzero,    $f_1(0):=f(1)=1$ and $f_1$ is eventually periodic by the assumption on $f$. Set $f_2(n'):=f(n')$ for $n'\in p^{\perp}$ and $f_2(p^k)=0$ for every $k>0$. Then $f(n)=f_1(\nu_p(n))f_2\Big(\frac{n}{p^{\nu_p(n)}}\Big)$ for every $n\in \N$. It is enough to prove that $f_2$ is periodic. By Theorem \ref{th:chiporaz1}, $\spec(Per(f))$ is finite, say $\spec(Per(f))=\{q_1,\ldots,q_u\}$ and assume that $q_1=p$. Let $m$ be a period of $f$ associated to the position 1 such that $\spec(m)=\spec(Per(f))$ and denote $b_j:=\nu_{q_j}(Per(f))$. Note that $b_j<\infty$ for $j=2,\ldots,u$.
We claim that $f_2$ is $T$-periodic, where $T:=p^{\nu_p(m)}q_2^{b_2}\ldots q_u^{b_u}$. Clearly, as $p|T$, $p^{\perp}+T=p^{\perp}$, so $f_2(n+T)=0=f_2(n)$ for $n$ divisible by $p$. Assume now that $n\in p^{\perp}$ and write $n=q_2^{a_2}\ldots q_u^{a_u}n'$, where $n'\in (q_1q_2\ldots q_u)^{\perp}$. This means that $n'\in\spec(Per(f))^{\perp}= m^{\perp}$, so $n'\in Per_m(f)$ by Lemma~\ref{eq:periodof1}. By Lemma~\ref{lem:zamiast1.9}
$n\in Per_{M}(f)$, where $M=q_2^{a_2}\ldots q_u^{a_u}m$. Finally, by Lemma~\ref{lem:general} (applied first to $q_2,b_2,n,M$, then to $q_3,b_3,n,q_2^{b_2-\nu_{q_2}(M)}M$, etc.), we conclude that
$n\in Per_{q_2^{b_2-\nu_{q_2}(M)}\ldots q_u^{b_u-\nu_{q_u}(M)}M}(f)$.
However, this yields
$n\in Per_{T}(f)$ because
$$
T=q_2^{b_2-\nu_{q_2}(M)}\ldots q_u^{b_u-\nu_{q_u}(M)}M,
$$
as $\spec(m)=\spec(M)=\{q_1,\ldots,q_u\}$ and $\nu_p(m)=\nu_p(M)=\nu_p(T)$. Moreover, $f(n)=f_2(n)$ and $f(n+T)=f_2(n+T)$ since $n,n+T\in p^{\perp}$. Therefore, $T$ is a period of $f_2$ and the proof is complete.

\section{When a pretentious function $f$ has precisely one Furstenberg system?}
The problem in the title of this section is equivalent to the fact that $f$ is generic or that all correlation limits of $f$ exist.
Prior to answering this question in Theorem \ref{t:OK}, we discuss a few relevant examples.
\subsection{Motivating examples}\label{s:nit} In the examples that we considered so far we have encountered precisely one Furstenberg system (in fact, in most cases, the sequences generated  uniquely ergodic subshifts). However, in general, we have uncountably many Furstenberg systems.
We just recall that each of the non-trivial Archimedean characters $n^{it}$, $t\neq0$, has uncountably many different Furstenberg systems because such functions do not have the mean, as a consequence of theorem of Hal\'asz .\footnote{All of the Furstenberg systems, as measure-theoretic systems, are isomorphic to the identity on the circle (considered with Lebesgue measure) \cite{Go-Le-Ru}.} Note however that a small perturbation of an Archimedean character:
$$
f(n)=n^{it}\text{ for }n\text{ odd}, f(2^k)=-2^{ikt}\text{ for }k\geq1,$$
(so $f(n)=g(n)n^{it}$ with $g(n)=(-1)^{n+1}$) yields a multiplicative function such that $\mathbb{D}(n^{it},f)<+\infty$ and $M(f)=0$ by Hal\'asz theorem. Nevertheless, a short argument shows that the second order correlation
$$
\lim_{N\to\infty}\frac1N\sum_{n<N}f(n)f(n+1)\text{ does not exist}.
$$
Indeed,
$$
|n^{it}(n+1)^{it}-n^{2it}|=|n^{2it}(e^{it\log(1+\frac1n)}-1)|=
|e^{it(\frac1n+O(1/n^2))}-1|=o(1),
$$
so,
$$
\lim_{N\to\infty}\frac1N\sum_{n<N}f(n)f(n+1)=
\lim_{N\to\infty}\frac1N\sum_{n<N}n^{2it}g(n)g(n+1)=$$$$
-\lim_{N\to\infty}\frac1N\sum_{n<N}n^{2it}$$
since $g(n)g(n+1)=-1$ for all $n\geq1$. Since the latter limit does not exist, our claim follows.

It is also possible for $f\sim \chi$ to have uncountably many Furstenberg systems, as this condition does not guarantee that $M(f)$ exists when $\chi$ coincides with the trivial character. Indeed, for the completely multiplicative $f$ defined by $f(p)=e^{2\pi i/\log\log p}$, where $f\sim 1$ and the mean does not exist (see \cite{Fr-Le-Ru}).

\subsection{Characterization of multiplicative RAP}
Let us recall that a sequence $f\in\mathbb{U}^{\Z}$ is called {\em Besicovitch
almost periodic} if for each $\vep>0$ there exists a trigonometric polynomial $P$ such that
$$
\|f-P\|_B:=\limsup_{N\to\infty}\frac1N\sum_{n<N}|f(n)-P(n)|<\vep.$$
By the {\em spectrum} of such a function it is meant the set of all non-zero values given by the limits (which do exist) $\lim_{N\to\infty}\frac1N\sum_{n<N}f(n)e^{-2\pi in\theta}$ for $\theta\in\T$ (if $f$ is additionally multiplicative then for all $\theta$ irrational the limit equals 0, see Theorem~1 \cite{Da-De}). By Theorem 2.7 \cite{Be-Ku-Le-Ri}, for all such $f$, the spectrum is empty iff $\|f\|_B=0$.\footnote{\label{f:btrivial} Assume additionally that $f\in\mathcal{M}$. Then, $\|f\|_B=0$ implies obviously that $f$ is aperiodic, whence, it is not pretentious. Moreover, $\|f\|_B=0$ iff $\mathbb{D}(|f|,1)=\infty$, Lemma 2.9 \cite{Be-Ku-Le-Ri}. It follows that pretentious multiplicative functions which are Besicovitch almost periodic have non-empty spectrum.}
When the approximates $P$ of a function $f$ have rational coefficients, we go back to the notion of RAP (cf.\ Footnote~\ref{f:rap}).
It has been proved in \cite{Be-Ku-Le-Ri} (see Corollary 2.11 therein) that for $f\in\mathcal{M}$ the notions of Besicovitch almost periodicity and RAP are equivalent.
Given a subsequence $(N_k)$, we say that $f$ is RAP along $(N_k)$ if,
for each $\vep>0$, there exists a periodic sequence $f_{\vep}$ such that
$$ \limsup_{k\to\infty}\frac1{N_k}\sum_{n<N_k}|f(n)-f_{\vep}(n)|<\vep.$$
If $f\in\mathcal{M}$ and  $f\sim \chi$ then it has been proved in \cite{Fr-Le-Ru} that in this case for each subsequence $(N_k)$ we can choose a further subsequence $(N'_k)$ along which $f$ is RAP.

Before we begin the proof of Theorem~\ref{t:OK}, we make a remark on complex-valued measures on the circle.

\begin{Remark}\label{r:spectral} Assume that $\sigma$ is a positive (finite), Borel measure on $\T$ and let $\sigma'\ll \sigma$ be a complex measure. If $\sigma$ is discrete then $(\widehat{\sigma'}(n))$ cannot be eventually zero (i.e.\ we cannot have $\widehat{\sigma'}(n)=0$ for $|n|\geq n_0$). Indeed, suppose that $\widehat{\sigma'}(n)=0$ for $|n|\geq n_0$. Then, $\sigma'=P(z)dz$, where $P(z)=\sum_{|k|<n_0}\widehat{\sigma'}(k)z^k$   for $|n|\geq n_0$. On the other hand, $\sigma'=\nu_1-\nu_2+i\nu_3-i\nu_4$, where all $\nu_j$ are positive (finite) measures with disjoint supports. Since  $\nu_j\ll\sigma$, they are purely atomic and  we obtain an obvious contradiction.\end{Remark}

{\em Proof of Theorem~\ref{t:OK}}.
(iii) $\Leftrightarrow$ (ii)
 This follows by combining Corollary 2.11 \cite{Be-Ku-Le-Ri} (which says that RAP is equivalent to Besicovitch almost periodicity for $f\in\mathcal{M}$) and Theorem~6 \cite{Da-De} (which says that $f\in\mathcal{M}$ is Besicovitch almost periodic with non-empty spectrum iff condition~\eqref{mlok} is satisfied), see also Theorem 2.6 in  \cite{Fr-Le-Ru}.

(iii) $\Rightarrow$ (i) The claim follows immediately since RAP is generic (see Theorem 1.7 \cite{Be-Ku-Le-Ri0}).

(i) $\Rightarrow$  (ii)
First of all, notice that (i) for $f$ implies (i) for any power $f^m$ (as (i) for $f$ is equivalent to the existence of all correlations involving {\bf all} powers $f^{a_i}$, $\ov{f}^{b_j}$).  We now claim that (i) for $f$ implies that $t=0$. Indeed, suppose $t\neq0$ and let $q$ be the modulus of $\chi$. We have (using the inequality $\mathbb{D}(f^m,g^m)\leq m\mathbb{D}(f,g)$)
$$
f^{\phi(q)}\sim n^{i\phi(q)t}.$$
Since the mean of $f^{\phi(q)}$ must exist (as $f^{\phi(q)}$ satisfies (i)),
\beq\label{ha1}
f^{\phi(q)}(2^k)=-2^{ki\phi(q)t}\text{ for all }k\geq1\eeq
in view of the Hal\'asz theorem.
But we can repeat the same reasoning for $2\phi(q)$ instead of $\phi(q)$ to obtain
$$
f^{2\phi(q)}(2^k)=-2^{ki2\phi(q)t}\text{ for all }k\geq1
$$ which is in contradiction with \eqref{ha1}.

Hence, $t=0$, so $\mathbb{D}(f,\chi)<+\infty$ and let $q$ denote the modulus of $\chi$. By our standing assumption~(i), we know that all correlations exist. In particular, all correlations of order~2 do and we claim that there exists $a\in\N$, $a>q$, such that
\beq\label{OK1}
\lim_{N\to\infty}\frac1N\sum_{n<N}f(n)f(n+a)\neq0.\eeq
Indeed, let $(X_f,\nu,S)$ denote the unique Furstenberg system of $f$. Let $\pi_0:X_f\to\C$ stand for the 0-coordinate projection.  Then, for each $k\in\Z$,
$$
\lim_{N\to\infty}\frac1N\sum_{n<N}f(n)f(n+k)=
\widehat{\sigma}_{\pi_0,\ov{\pi}_0}(k),$$
where $\sigma_{\pi_0,\ov{\pi}_0}$ is a complex (spectral) measure determined by $$\widehat{\sigma}_{\pi_0,\ov{\pi}_0}(k)=\int_{X_f}\pi_0\cdot\pi_0\circ S^k\,d\nu=\int_{\T}z^k\,d\sigma_{\pi_0,\ov{\pi}_0},\;k\in\Z.$$
By the general spectral theory of unitary operators on separable Hilbert spaces, $\sigma_{\pi_0,\ov{\pi}_0}\ll\sigma_{\pi_0}$. Moreover, by the general theory of Furstenberg systems of pretentious functions \cite{Fr-Le-Ru}, we know that $\sigma_{\pi_0}$ is purely atomic. Our claim now follows from Remark~\ref{r:spectral}.

By the correlation formulas in Theorem~1.5 in~\cite{Kl} (with $\ov{F}$ replaced by $F$; the same proof works as noted in \cite{Kl}), we have
$$
\sum_{n<N}f(n)f(n+a)=N\prod_{p\leq N} \mathscr{M}_p+o(N).$$
We define another multiplicative function $F$ by setting $F(p^k):=f(p^k)\ov{\chi(p^k)}$ if $p$ does not divide the modulus $q$ and $F(p^k)=1$ if $p|q$.
Now, since the infinite product converges (see~\eqref{OK1}), all the local factors $\mathscr{M}_p$ are different from zero. Furthermore, the number of local factors $\mathscr{M}_p$ such that
$$
p^\ell||q\text{ or }(p^0||q\wedge p|a)$$
is finite. For all remaining local factors, whence, for all large $p\gg a$,  the local factors ($n=0$ in the notation of Thm.\ 1.5 in \cite{Kl}, $p^0||a$) are:
$$
\mathscr{M}_p=1-\frac2{p^{0+1}}+\Big(1-\frac1p\Big)
\cdot 2\sum_{j>0}\frac{F(p^0)F(p^j)}{p^j}=$$$$
1-\frac2p+2\Big(1-\frac1p\Big)\Big(\frac{F(p)}{p}+\frac{F(p^2)}{p^2}+\ldots\Big)=$$$$
1-\frac2p+\frac{2F(p)}p+O\Big(\frac1{p^2}\Big)=1-2\frac{1-F(p)}{p}+
O\Big(\frac1{p^2}\Big).$$
Now, the infinite product $\prod_{p\gg a}\mathscr{M}_p$ must converge, so by taking the logarithm, the series $\sum_p\frac1p(1-f(p)\ov{\chi(p)})$ converges which finishes the proof of the claim.\bez


\begin{Cor}\label{c:OK1} There is no pretentious function  having a unique {\bf non-ergodic} Furstenberg system.\end{Cor}
\begin{proof} This follows from Theorem~\ref{t:OK} since the unique system is an ergodic odometer.\end{proof}

\begin{Cor} If $f\sim \chi\cdot n^{it}$, $t\neq 0$, then $f$ has uncountably many
different Furstenberg systems. All of them are isomorphic and each of them has uncountably many ergodic components.\end{Cor}
\begin{proof} This follows from Corollary~\ref{c:OK1} and \cite{Fr-Le-Ru}, where it has been proved that whenever $t\neq 0$ then no Furstenberg is ergodic (and that for pretentious functions all Furstenberg systems are isomorphic).\end{proof}

\begin{Cor}\label{c:FV0}
If $f$ takes only finitely many values then condition \eqref{mlok} is equivalent to
\beq\label{mlok1} \sum_{p\colon f(p)\neq \chi(p)}\frac1p<+\infty.\eeq
\end{Cor}

\subsubsection{The FH-conjecture}
Let us now discuss Conjecture~1 from \cite{Fr-Ho}: {\em Every real-valued $f\in\mathcal{M}$ has a unique Furstenberg system (i.e.\ $f$ is a generic point)}.

\begin{Cor} \label{c:FHc1}Frantzikinakis-Host's conjecture (FH-conjecture) is true in the class of pretentious functions.\end{Cor}
\begin{proof}
Let $f:\N\to[-1,1]$ be multiplicative. Assume that $\mathbb{D}(f,\chi\cdot n^{it})<+\infty$ (with $\chi$ of modulus~$q$). Then $t=0$ since $\mathbb{D}(f^{\phi(q)},n^{i\phi(q)t})<+\infty$.\footnote{Note that if $g\in\mathcal{M}$ is real-valued then
since ${\rm Re}(p^{it})={\rm Re}(p^{-it})$ and $g(p)\in\R$, we have
$$
\mathbb{D}(1,n^{2it})=\mathbb{D}(n^{-it},n^{it})\leq 2\mathbb{D}(g,n^{it})$$
(using the triangle inequality for $\mathbb{D}$ to get the latter inequality),
so if $t\neq0$ then $\mathbb{D}(g,n^{it})=+\infty$; apply this to $g=f^{\phi(q)}$.}
Hence $\mathbb{D}(f,\chi)<+\infty$ and since $f$ is real-valued, also $\chi$ has to be a real character (i.e.\ taking only $0,\pm1$-values).\footnote{If $0\neq a\neq\pm1$ and there is $p$ so that $f(p)=a$ then the set $\{p'\colon f(p')=a\}$ contains a set $r+q\N$ with ${\rm gcd}(r,q)=1$, so (in view of Dirichlet's theorem)
$$\sum_{p'\colon f(p')=a}\frac1{p'}(1-f(p'){\rm Re}(a))=+\infty$$
since $|{\rm Re}(a)|<1$.}
It follows that
$$
\sum_{p}\frac1p(1-f(p)\chi(p))=\sum_p \frac1p(1-{\rm Re}(f(p)\chi(p)))<+\infty.$$
The result now follows from Theorem~\ref{t:OK}.\end{proof}

\begin{proof} {\em of Corollary~\ref{c:FH10}.}
Let $f\in\mathcal{M}$, $f:\N\to[-1,1]$. We need to show the existence of limits for all possible correlations:
\beq\label{cor111}
\lim_{N\to\infty}\frac1N\sum_{n<N}\prod_{j=1}^kf(n+b_j)^{a_j}\eeq
for all $k\geq1$, all possible integers $b_1<\ldots<b_k$ and all $a_1,\ldots,a_k\in\N$.

{\bf Case 1.} If $f$ is pretentious then the existence of all limits~\eqref{cor111} follows from Corollary~\ref{c:FHc1}.

{\bf Case 2.} If $f$ is aperiodic and all numbers $a_1,\ldots,a_k$ are even, then remembering that $f^2$ is pretentious (in fact, it is RAP, see Remark 2.12 in \cite{Be-Ku-Le-Ri}), the result follows.

{\bf Case 3.} Assume that $f$ is aperiodic and, for some $1\leq j_0\leq k$, we have $a_{j_0}$ is odd. Then, we claim that $f^{a_{j_0}}$ remains aperiodic. Indeed, $f^{a_{j_0}}=fg$, where $g=f^{a_{j_0}-1}$, so $g\geq0$ and therefore (see Case 2.) $f$ is the product of an aperiodic and a RAP multiplicative function. All such products remain aperiodic by Proposition 2.23 \cite{Be-Ku-Le-Ri}.\footnote{The result follows from a Frantzikinakis-Host's theorem \cite{Fr-Ho0} on the equivalence of notions of aperiodicity and uniformity (in the class $\mathcal{M}$) and some basic properties of GHK semi-norms.}   Finally, by Appendix~C in \cite{Ma-Ra-Ta}, $f^{a_{j_0}}$ is strongly aperiodic. A use of the cE-conjecture shows that the limit in \eqref{cor111} exists (it equals~0).
\end{proof}


Finally, we will show that Theorem~\ref{t:OK} applies when $f$ takes finitely many values.

\begin{Lemma}\label{l:FV} Let $f\in\mathcal{M}$ takes finitely many values and let $f\sim n^{it}$. Then $t=0$.\end{Lemma}
\begin{proof} By assumption, $\{f(p)\colon p\in\PP\}=\{a_1,\ldots,a_k\}$ and let $P_j:=f^{-1}(a_j)$, $j=1,\ldots,k$. By renaming the values if necessary we can assume that $\PP=P_1\cup \ldots\cup P_s\cup P_{s+1},\ldots, P_k$, where $P_j$ is finite for $j\leq s$ and $P_j$ is infinite if $j>s$. It follows that for $j>s$, $a_j$ must be a root of unity.\footnote{Because, if $p_{r_1},\ldots,p_{r_L}\in P_j$ then
$f(p_{r_1}\cdot\ldots\cdot p_{r_L})=a_j^L$, and $f$ takes only finitely many values.} Hence $\PP=P\cup Q$, where $P$ is finite and, for some $m\geq1$ and all $p\in Q$, $f(p)^m=1$. Now,
$$
+\infty >m^2 \mathbb{D}(f, n^{it})^2\geq \mathbb{D}(f^m,n^{imt})^2=$$$$
\sum_{p\in P}\frac1p(1-{\rm Re}(f(p)^mp^{imt}))+\sum_{p\in Q}\frac1p(1-{\rm Re}(f(p)^mp^{imt}))=$$$$
O(1)+\sum_{p\in Q}\frac1p(1-{\rm Re}(p^{imt}))= O(1)+(O(1)+\mathbb{D}(1,n^{imt})^2).$$ Thus $\mathbb{D}(1,n^{imt})<+\infty$, and therefore $t=0$.\end{proof}
It follows from Lemma~\ref{l:FV} that if $f\in\mathcal{M}$ takes finitely many values and $f\sim n^{it}\cdot\chi$ then $t=0$. Indeed, $f\cdot\ov\chi$ still takes finitely many values.

\begin{Cor}\label{c:FV} If $f\in\mathcal{M}$ is pretentious and takes finitely many values then $f$ has a unique Furstenberg system (and it is ergodic). In fact, $f$ is RAP.\end{Cor}
\begin{proof} By the above, $f\sim\chi$, that is,
\beq\label{fv1}
\sum_p \frac1p\big(1-{\rm Re}(f(p)\ov{\chi(p)})\big)<+\infty.\eeq
However, since $f$ takes finitely many values, so does $f\cdot\chi$ and the values are in $\mathbb{U}$. Therefore, in order to have \eqref{fv1} satisfied we must have
$$\sum_{p\in\PP, f(p)\neq \chi(p)}\frac1p<+\infty.$$
The claim follows from Corollary~\ref{c:FV0} and Theorem~\ref{t:OK}.
\end{proof}

\subsection{More on RAP multiplicative functions}
We will now answer the question asked in Remark~\ref{r:ml1}.
Can we find an infinite $F\subset \mathbb{P}$, $\sum_{p\in F}\frac1p<+\infty$ satisfying $f(n)=\chi(n)$ for $n$ coprime with $F$  (note that these conditions imply $f\sim \chi$) such that the mean of $f$ {\bf does not exist}? The answer turns out to be negative (cf. Lemma 4.6 in \cite{Be-Ku-Le-Ri} and its proof).

\begin{Cor}\label{c:male1}
Assume that $f\in\mathcal{M}$ and let $F\subset \mathbb{P}$, $\sum_{p\in F}\frac1p<+\infty$  be an infinite set satisfying $f(n)=\chi(n)$ for $n$ coprime with $F$. Then $f$ is RAP and has a unique Furstenberg system. In particular, $M(f)$ exists.\end{Cor}
\begin{proof}
We have
$$
\sum_{p\in\mathbb{P}}\frac1p(1-f(p)\ov{\chi(p)})=O(1)+
\sum_{p\in F}\frac1p(1-f(p)\ov{\chi(p)})$$
and the latter series is (even) absolutely convergent. The result follows from Theorem~\ref{t:OK}.\end{proof}

\section{GHK semi-norms for pretentious functions}\label{s:ghk} Recall that given $f$ and $(N_k)$  which defines a Furstenberg system of $f$, the Gowers-Host-Kra (GHK) semi-norms are defined as (see, e.g. \cite{Ho-Kr}):
$$
\|f\|^2_{u^1}=\|f\|^2_{u^1((N_k))}=
\lim_{H\to\infty}\frac1H\sum_{h\leq H}\Big(\lim_{k\to\infty}\frac1{N_k}\sum_{n\leq N_k}f(n)\ov{f(n+h)}\Big),$$
and, for $s>1$,
$$
\|f\|^{2^{s+1}}_{u^{s+1}}=\|f\|^{2^{s+1}}_{u^{s+1}((N_k))}
=\lim_{H\to\infty}\frac1H\sum_{h\leq H}\| f(\cdot+h)\cdot\ov{f(\cdot)}\|^{2^s}_{u^s((N_k))}.$$
\begin{Prop}\label{p:ghk1}
If $f\in\mathcal{M}$ is pretentious then $\|f\|_{u^2}>0$ (for some $(N_k)$ defining a Furstenberg system of $f$).\end{Prop}
\begin{proof} A dynamical proof follows from \cite{Ka-Ku-Le-Ru}, we will describe its main steps. Suppose that $\|f\|_{u^2}=0$ (for all relevant $(N_k)$).
Then, Corollary 1.10 \cite{Ka-Ku-Le-Ru} characterizes arithmetic functions having the second GHK-semi-norm equal to zero as those orthogonal to all topological systems whose all {\bf ergodic} measures give systems with discrete spectrum. Since we deal with so called ec-characteristic class, each (topological) dynamical system for which  all ergodic measures yield systems with discrete spectrum enjoys the strong $f$-MOMO property (see Proposition 2.17 in \cite{Ka-Ku-Le-Ru}). Take any Furstenberg system $(X_f,\kappa,S)$ of $f$ ($\kappa$ yields a system with discrete spectrum, but also all ergodic components of $\kappa$ yield discrete spectrum systems by \cite{Fr-Le-Ru}) and consider its any Hansel model $(Y,\kappa',S')$ (in this model, measure-theoretic dynamical systems coming from ergodic measures are isomorphic to systems given by ergodic components of $\kappa$). Then, this Hansel model satisfies strong $f$-MOMO property, which however yields
a contradiction with Proposition~5.6 \cite{Ka-Ku-Le-Ru}.
\end{proof}

It follows that all $u^s$-GHK-semi-norms of a pretentious $f$ are positive for all $s\geq2$. It remains to decide when  the first GHK-semi-norm is zero (we recall that all Archimedean characters have this seminorm positive (e.g.\ \cite{Go-Le-Ru})). Theorem A.1. in \cite{Ma-Ra-Ta} tells us that this semi-norm vanishes for all $f\in\mathcal{M}$ which do not pretend being Archimedean character. When we restrict to pretentious functions then we have a complete characterization which we prove using ergodic tools.

\begin{Prop}\label{p:ghk2} If $f\in\mathcal{M}$  is pretentious then $\|f\|_{u^1}=0$ if and only if $M(f)=0$.\end{Prop}
\begin{proof}The necessity is obvious, since $\|f\|_{u^1}=0$ iff $\mathbb{E}^\kappa(\pi_0|{\rm Inv}_\kappa)=0$ (Inv$_\kappa$ stands for the $\sigma$-algebra of invariant sets modulo $\kappa$) for each Furstenberg system $\kappa$ of $f$, in particular, $\int_{X_f}\pi_0\,d\kappa=0$.

To prove the sufficiency, we assume that $f\sim n^{it}\cdot \chi$. Write
$$f=g_1\cdot g_2\text{ with }g_1=n^{it},\; g_2=f\cdot n^{-it}.\footnote{Remembering that $M(f)=0$, note that\\
(i) if $\mathbb{D}(f,n^{is})=+\infty$ for all $s\in\R$, then the same holds for $g_2$ and whence $M(g_2)=0$ by Hal\'asz theorem, while\\
(ii) if $\mathbb{D}(f,n^{iu})<+\infty$ and $f(2^k)=-2^{kiu}$ (for all $k$) then $\mathbb{D}(g_2,n^{i(u-t)}))<+\infty$ and $g_2(2^{k})=-2^{ki(u-t)}$ for each $k\geq1$. It follows from Hal\'asz theorem that $M(g_2)=0$.}$$
Note that $g_2\sim \chi$, so by \cite{Fr-Le-Ru}, all Furstenberg systems of $g_2$ are ergodic and  $\int \pi_0\,d\kappa_2=0$ for each Furstenberg system $\kappa_2$ of $g_2$ (by our standing assumption $M(f)=0$ which forces $M(g_2)=0$). It follows that the spectral measure
\beq\label{er1}\sigma_{\pi_0,\kappa_2}\text{ has no atom at }1\eeq
for each Furstenberg system $\kappa_2$.
Assume that $(N_k)$ is given\footnote{Following Proposition 4.1, see also Remark 4.2, both in \cite{Go-Le-Ru}, we aim at showing that for each Furstenberg system $\kappa$ for $f$, we have $\mathbb{E}^\kappa(\pi_0|Inv_\kappa)=0$, equivalently, that the spectral measure $\sigma_{\pi_0,\kappa}$ has no atom at~1. We show below that each Furstenberg system of $f$ is given by the $M$-image (coordinatewise multiplication) of the Cartesian product of Furstenberg systems of $g_1$ and $g_2$.} and that
$$
\frac1{N_k}\sum_{n<N_k}\delta_{g_i\circ S^n}\to\kappa_i,\;i=1,2.$$
Then $\kappa_1$ yields an identity system \cite{Go-Le-Ru}, while $\kappa_2$ yields an ergodic system, so the two systems $(S,\kappa_1)$ and $(S,\kappa_2)$ are disjoint. It follows that
$$
\frac1{N_k}\sum_{n<N_k}\delta_{(g_1,g_2)\circ (S\times S)^n}\to\kappa_1\ot\kappa_2.$$
Hence, $f=g_1\cdot g_2$ is generic along $(N_k)$ for the measure $M_\ast(\kappa_1\otimes \kappa_2)$, where $M$ stands for the coordinatewise multiplication. Then, for each $m\in\Z$,
$$
\widehat{\sigma}_{\pi_0,M_\ast(\kappa_1\ot\kappa_2)}(m)=
\widehat{\sigma}_{\pi_0,\kappa_1}(m)\cdot \widehat{\sigma}_{\pi_0,\kappa_2}(m),$$ and therefore
$$\sigma_{\pi_0,M_\ast(\kappa_1\ot\kappa_2)}=
\sigma_{\pi_0,\kappa_1}\ast \sigma_{\pi_0,\kappa_2}.$$
But
$\sigma_{\pi_0,\kappa_1}=\delta_1$, so
$\sigma_{\pi_0,M_\ast(\kappa_1\ot\kappa_2)}=
\sigma_{\pi_0,\kappa_2}$. By  \eqref{er1}, it follows that
$$
\sigma_{\pi_0,M_\ast(\kappa_1\ot\kappa_2)}\text{ has no atom at }1.$$
Since $M(f)=0$, the result follows.
\end{proof}

Note that $f(n)=n^{it}$ for $n$ odd and $f(2^k)=-2^{kit}$ (considered in Section~\ref{s:nit}) is an example of $f\sim n^{it}$ for which $\|f\|_{u^1}=0$.


\section{Relations between various subclasses of aperiodic multiplicative functions}\label{s:relacje}
Our goal in the next few sections is to clarify various relations between subclasses of bounded multiplicative functions.
\subsection{Liouville-like functions}
\begin{Prop}\label{p:LMA}
If $f\in\mathcal{M}$ is Liouville-like (see Section~\ref{s:Llike}) then $f$ is moderately aperiodic and strongly aperiodic.
\end{Prop}
\begin{proof} First of all $f$ is aperiodic, as shown in \cite{Kl-Ma-Te} (aperiodicity also follows from the fact that $f$ satisfies the Chowla conjecture). Let $A>0$. We want to prove that
$$\lim_{N\to\infty}\frac1{\log\log N}\inf_{|t|<N^A}\min_{\chi~\text{mod}~q, q\leq(\log N)^A}\sum_{p\leq N}\frac1p(1-{\rm Re}(f(p)\chi(p)p^{it}))=0.$$
Consider $t=0$, $\chi=1$. We then have
$$
\lim_{N\to\infty}\frac1{\log\log N}\sum_{p\leq N} \frac1p(1-{\rm Re}(f(p)))=$$$$
\lim_{N\to\infty}\frac{\sum_{p\leq N, p\in\mathcal{P}}\frac2p}{\log\log N}=
\lim_{N\to\infty}\frac{\sum_{p\leq N, p\in\mathcal{P}}\frac2p}{\sum_{p\leq N}\frac1p}=0$$
where the last equality follows by the density assumption on $\mathcal{P}$. Therefore, $f$ is moderately aperiodic. It is also strongly aperiodic because it is aperiodic and real-valued \cite{Ma-Ra-Ta}.\end{proof}

\begin{Remark}\label{r:trma} The proof of course shows that each aperiodic function $f\in \mathcal{M}$ such that $f(p)=1$ on a set $\mathcal{Q}\subset\mathbb{P}$ of full density in $\mathbb{P}$ is moderately aperiodic. Note that despite the fact that in this case $\mathbb{P}\setminus\mathcal{Q}$ has zero density, it can be ``large'' in the sense that $\sum_{p\in \mathbb{P}\setminus\mathcal{Q}}\frac1p=+\infty$.\end{Remark}

\subsection{Trivial aperiodic functions}
We consider now $f\in\mathcal{M}$ for which $\|f\|_B=0$, we call them {\em trivial} (all such functions are aperiodic).
\begin{itemize}
\item All trivial functions are strongly aperiodic.

Indeed, using Footnote~\ref{f:btrivial},
$$\inf_{|t|\leq N}\sum_{p\leq N}\frac1p\big( 1-{\rm Re}(f(p)\chi(p)p^{it})\big)\geq \sum_{p\leq N}\frac1p\big(1-|f(p)|\big)\to\infty$$
when $N\to\infty$.

\item There are trivial functions which are not moderately aperiodic.

Indeed, if we assume that $|f(p)|\leq 1-\eta_0$ for some $\eta_0>0$ and all $p\in\PP$ then
$$
\inf_{|t|\leq N^A}\min_{\chi\text{ mod }q;q\leq(\log N)^A}\sum_{p\leq N}\frac1p\big( 1-{\rm Re}(f(p)\chi(p)p^{it})\big)\geq$$$$
\sum_{p\leq N}\frac1p\big( 1-|f(p)|\big)\geq \frac{\eta_0}2\log\log N$$
and the claim follows.

\item There are trivial functions which are moderately aperiodic.

Indeed, let $\mathcal{Q}\subset \PP$ be a subset of full density in $\PP$ satisfying $\sum_{p\in\PP\setminus \mathcal{Q}}\frac1p=+\infty$. Set $f(p)=1$ for $p\in \mathcal{Q}$ and $f(p)=0$ for the remaining $p$. We have
$$
\sum_{p\in\PP}\frac1p(1-|f(p)|)=+\infty,$$
so $\|f\|_B=0$. It follows from Remark~\ref{r:trma} that $f$ is moderately aperiodic.
\end{itemize}

\subsection{Liouville function is not moderately aperiodic}\label{sec:Liouville_not_MA}
It has already been noticed in \cite{Kl-Ma-Te} that the conditions defining moderately aperiodic  and strongly aperiodic functions seem to be ``independent''. However, no explicit examples are provided. We show below that the Liouville function is strongly aperiodic but it is not moderately aperiodic and leave open the question of whether there exists a moderately aperiodic function which is not strongly aperiodic.

\begin{Prop}\label{p:kmt7}
The Liouville function $\la$ is not moderately aperiodic.\end{Prop}
Before we start the proof, let us recall the following result which follows from Lemma 6.1 from \cite{Kl-Ma-Te} (and the remark just after this lemma):

\begin{Lemma}\label{l:kmt8} If $f\in\mathcal{M}$ is moderately aperiodic then there exist a Dirichlet character $\chi$, $t\in\R$ and $\eta=\eta(x)>0$ with $\eta(x)\to0$ when $x\to\infty$, such that
$$
\mathbb{D}(f,\chi\cdot n^{it};x^{\eta(x)},x)^2\leq \eta (x)\text{ for infinitely many }x\to\infty.$$
\end{Lemma}

We recall that for $g,h\in\mathcal{M}$,
$$
\mathbb{D}(g,h; x,y)^2=\sum_{x\leq p<y}\frac1p(1-{\rm Re}(g(p)\ov{h(p)}))$$
and that
for this ``distance'' we have the triangle inequality, and
\beq\label{mno}
\mathbb{D}(gg',hh'; x,y)\leq \mathbb{D}(g,h; x,y)+ \mathbb{D}(g',h'; x,y),\eeq
see Lemma 3.1 in \cite{Gr-So0}.

\vspace{2ex}

\noindent
{\em Proof of Proposition~\ref{p:kmt7}}. Suppose that $\la$ is moderately aperiodic. In view of Lemma~\ref{l:kmt8}, there exist $\chi,t,\eta$ with $X_k\to\infty$ such that
\beq\label{kmt9}
\sum_{X_k^{\eta(X_k)}\leq p\leq X_k}\frac1p(1+{\rm Re}(\chi(p)p^{it}))\leq \eta(X_k).\eeq

{\bf Case 1.} $\chi=1,t=0$. In this case~\eqref{kmt9} reads (using Mertens' theorem)
$$
\sum_{X_k^{\eta(X_k)}\leq p\leq X_k}\frac2p=2(\log\log X_k+C+o(1)-(\log\log X_k^{\eta(X_k)}+C+o(1)))=$$$$
-2\log\eta(X_k)+o(1)\to+\infty,$$
which is a contradiction. (Note that the same argument applies if $\chi$ is principal since $\chi(p)=1$ for $p>N_0$.)

{\bf Case 2.} $\chi=1, t\neq0$.
We consider the intervals $I$ for which we have $1+{\rm Re}\, p^{it}\geq1/2$ for $p\in I$; equivalently, we require
$$
\cos(t\log p)\geq-\frac12.$$
Note that this is the case whenever
$$
\frac43\pi\leq t\log p\leq \frac{8}3\pi,$$
while if $t\log p$ is in the interval $(\frac23\pi,\frac43\pi)$, we have $\cos(t\log p)<-1/2$. By shifting the intervals by $2\pi$, we obtain the following relations:
If
$$
t\log p\in \Big(\frac23\pi+2r\pi,\frac43\pi+2r\pi\Big)\; \Longrightarrow\; 1+{\rm Re}\, p^{it}<\frac12;$$
$$
t\log p\in \Big[\frac43\pi+2r\pi, \frac{8}3\pi+2r\pi\Big]\;\Longrightarrow\; 1+{\rm Re}\,p^{it}\geq\frac12.$$
Denote
$$
J_r:=\Big(e^{\frac2{3t}\pi+2\frac rt\pi},e^{\frac4{3t}\pi+2\frac rt\pi}\Big),\;
I_r:=\Big[e^{\frac4{3t}\pi+2\frac rt\pi}, e^{\frac{8}{3t}\pi+2\frac rt\pi}\Big].$$
Note that the intervals we have defined are pairwise disjoint and their union covers a half-line contained in $[0,\infty)$. Moreover, their length is up to some multiplicative constants of the same order as the left-end (or right-end) point.
When $p\in I_r$ then $1+{\rm Re}\,p^{it}\geq\frac12$. By the PNT there exist $r_0$ and  a constant $D>0$ such that for all $r\geq r_0$ the number of primes in $I_r$ is at least $D\frac{e^{2\frac rt\pi}}{r/t}$. It follows that
\beq\label{ten}
\sum_{p\in I_r}\frac1p(1+{\rm Re}\, p^{it})\geq D' \frac{e^{2\frac rt\pi}}{r/t}\cdot \frac1{e^{\frac{8}{3t}\pi+2\frac rt\pi}}\geq D''\frac1r.\eeq

Now, take $k\geq 1$ large enough so that $\delta:=\eta(X_k)>0$ is very small and find the largest $r$ so that the left-end of $I_{r-1}$ is smaller than $X_k^\delta$, then the left-end of $I_r$ is $>X_k^\delta$. Note that $I_r\subset [X_k^\delta,X_k]$ because the lengths of intervals $J_r,I_r$ are of the same order as their end-point, hence are of the same order as $X_k^\delta$ (with $\delta$ small and $X_k$ very big). We now consider all intervals $I_r,I_{r+1},\ldots,I_{r+s}$ which are contained in $[X_k^\delta,X_k]$. We claim that $s\geq r$. Indeed,
the right-end point of $I_r$, that is,  $e^{\frac{8}{3t}\pi+2\frac{r}t\pi}$ is up to a (multiplicative) constant equal to $X_k^\delta$, then the right-end of $I_{r+j}$ is $e^{\frac{8}{3t}\pi+2\frac{r+j}t\pi}$, so by taking $j=r$, we obtain $e^{\frac{8}{3t}\pi+2\frac{2r}t\pi}$ which is up to a (multiplicative) constant $X_k^{2\delta}$, so much smaller than $X_k$. It now follows from \eqref{ten} that
$$
\sum_{X_k^\delta\leq p\leq X_k}\frac1p(1+{\rm Re}\,p^{it})\geq const. \sum_{j=1}^r\frac1{r+j}\geq const.>0$$
which yields a contradiction and therefore, completes the proof.

{\bf Case 3.} $\chi$ arbitrary and $t\neq0$.
Let $q$ be a modulus of $\chi$. If $\lambda$ is moderately aperiodic then in view of \eqref{mno},
$$
\mathbb{D}(\la^q,n^{iqt};x^{\eta(x)},x)\leq
q\mathbb{D}(\la,\chi\cdot n^{it};x^{\eta(x)},x),$$
so if $q$ is odd then $\lambda^q=\lambda$ and by considering $x=X_k$ we obtain that the upper bound for the l.h.s. is $q\eta(X_k)$ which leads to a contradiction as in the previous case. If $q$ is even, we obtain
$$
\mathbb{D}(1,n^{qit}; X_k^\delta,X_k)<\eta(X_k)$$
which leads to a contradiction by the same argument.

{\bf Case 4.} $\chi\neq1$ and $t=0$.
We apply the same argument since along an arithmetic progression $q\N+j$ (with $j$ and $q$ coprime) the Dirichlet character will be constant and the sum
$$
\sum_{X_k^\delta\leq p\leq X_k, p\in q\Z+j}\frac1p.$$
can be handled using relevant version of Mertens theorem for arithmetic progressions ($\sum_{p\leq N, p\in q\N+j}\frac1p\sim \frac1{\phi(q)}\log\log N$.)


\subsection{``MRT-modified" - a class of aperiodic functions which are neither strongly aperiodic nor moderately aperiodic}
We briefly outline the construction of modified MRT functions, which possess property from the title of this section. These functions satisfy the following conditions: there are  increasing sequences of natural numbers $(t_m)$, $(s_m)$ and $(\tilde{t}_m)$ such that
$$
t_m<s_{m+1}<s_{m+1}^2<\tilde{t}_{m+1}<t_{m+1}$$
and
$$f(p)=p^{is_{m+1}}\text{ for }p\in [t_m,\tilde{t}_{m+1}],\; f(p)=-1\text{ for }p\in [\tilde{t}_{m+1},t_{m+1}]$$ and for $p\in [1,t_m)$, we require that $$|f(p)-p^{is_{m+1}}|<1/t_m^2.$$ We can carry out the construction of such functions in exactly the same manner as in \cite{Ma-Ra-Ta}. For our purposes, we impose an extra assumption
\beq\label{doublelg}
t_{m+1}>e^{e^{\widetilde{t}_{m+1}}}.
\eeq
Notice that~\eqref{doublelg} implies that
\beq\label{lglg}
\text{ the upper double }\log-\text{density of }\bigcup_m[\widetilde{t}^{m+1}_{m+1},t_{m+1}]\text{ is positive}.\eeq

Then, we proceed as in the proof of Proposition 3.4 in \cite{Go-Le-Ru} to show that for each sequence $(N_m)$ for which $N_m<\widetilde{t}_{m+1}$ for all $m\geq1$, we have
\beq
\label{glr1}\lim_{m\to\infty}\frac1{N_m}\Big|\Big\{n\leq N_m\colon |f(n)-n^{is_{m+1}}|>\frac1{t_m}\Big\}\Big|=0.\eeq
Next,  for any $d\geq1$ and $\frac1{d+1}<\beta<\frac1d$, consider $N_m=[s_{m+1}^\beta]$, $m\geq1$. We can show that along this sequence we obtain a Furstenberg system, precisely the same as in \cite{Go-Le-Ru}, since in \cite{Go-Le-Ru} only~\eqref{glr1} is used and the analytic arguments used in \cite{Go-Le-Ru} do not depend on the primes in the interval $[\widetilde{t}_{m+1},t_{m+1}]$. It follows in particular that $f$ satisfies the Chowla conjecture along a subsequence, whence $f$ is aperiodic.

Now, note that $f$ cannot be strongly aperiodic by the same argument which we used to show that the MRT functions are not strongly aperiodic (using $X=\tilde{t}_{m+1}$, $t=s_{m+1}$ and $\chi=1$).

In order to show that $f$ is not moderately aperiodic, we can proceed as in the proof of Proposition~\ref{p:kmt7}, however we require $X_k$ to be in the region $[\tilde{t}^{m+1}_{m+1},t_{m+1}]$. Such a choice of $X_k$ is possible because  Lemma~\ref{l:kmt8} holds for the set of $x$ of full double-log density, see \cite{Kl-Ma-Te}, and \eqref{doublelg} holds.
Notice that if $X_k^{\eta(X_k)}>\widetilde{t}_{m+1}$ then our reasoning follows exactly the same lines of the proof of Proposition~\ref{p:kmt7}. If this inequality fails we first find the largest $r\geq1$ so that the left-end point of $I_{r-1}$ is $<\widetilde{t}_{m+1}$  and show that
the intervals $I_r,I_{r+1},\ldots, I_{2r}$ are all contained in $[\widetilde{t}_{m+1},t_{m+1}]$.

\section{Appendix}
\subsection{Product of automatic sequences need not be automatic}\label{s:przyklad}
Given two different primes $p,q$, we consider the multiplicative functions  $f(n)=(-1)^{\nu_p(n)}$, $f'(n)=(-1)^{\nu_q(n)}$. Note that we have~\eqref{ilo} with $\chi=1=\chi'$ and $F=\{p\}$, $F'=\{q\}$, so  the sequence
\beq\label{przyklad}
h(n):=(-1)^{\nu_p(n)+\nu_q(n)}, \;n\geq1,
\eeq
satisfies the BBC conjecture. However, $h$
is not automatic.

Indeed, a short analysis shows that $f$ (resp. $f'$) is a Toeplitz sequence with the period structure $(p^n)$ (resp. $(q^n)$). Hence (see Footnote~\ref{f:23})
the spectrum of any period structure of $h$ is contained in $\{p,q\}$. We show that both the $p$-valuation and the $q$-valuation of the periods of $h$ are infinite. Suppose otherwise, e.g. $b=\nu_p(Per(h))<\infty$. Then, by Lemma~\ref{lem:general}, the position $p^{b+1}$ has the associated period of the form $p^bq^c$ for some $c\in\N_0$. But there exists $k\in\N$ such that $p+kq^c$ is coprime with $p,q$ and therefore $(-1)^{b+1}=h(p^{b+1})=h(p^b(p+kq^c))=h(p^b)=(-1)^b$, a contradiction.  So the period structure of $h$ is $(p^nq^n$) and $h$ is not automatic in view of Proposition~\ref{prop:auto_Toeplitz}.


Another proof of the fact that $h$ is not automatic uses dynamical arguments. Indeed,
note that $f$ and $f'$ are both automatic and they determine uniquely ergodic systems $(X_f,\nu,S)$ and $(X_{f'},\nu',S)$. As measure-theoretic systems they have discrete spectrum with the group of eigenvalues  generated by the roots $e^{2\pi i/p^n}$ ($n\in\N_0$) for the first system by $e^{2\pi i/q^n}$ for the second one. Moreover, $f,f'$ are generic points for $\nu,\nu'$, respectively. As the measure-theoretic systems are disjoint in the Furstenberg sense, $(f,f')$ is a generic point for the product measure $\nu\ot\nu'$. Then $h=f\cdot f'$ is a generic point for the convolution $\nu\ast\nu'$, and the system
$(X_{h},\nu\ast\nu',S)$ is a uniquely ergodic factor (via the coordinatewise multiplication map) of the Cartesian product
$$(X_{f}\times X_{f'},\nu\otimes\nu',S\times S),$$
with the latter system being uniquely ergodic (by disjointness).
Now, if $h$ were automatic, it would have to be $r$-automatic for some $r\in\mathbb{P}$ by \cite{Ko-Le-Mu}. Then, by the general theory of (primitive) automatic sequences, the spectrum of $(X_h,\nu\ast\nu',S)$ being
the group generated by $e^{2\pi i/r^n}$ (modulo a finite group of roots of unity) must be a subgroup of the group generated by the eigenvalues of $(X_f,\nu,S)$ and $(X_{f'},\nu',S)$, which is possible only if $r=p$ or $r=q$. In the first case, it would mean that the factor (corresponding to $\nu\ast\nu'$) given by the eigenfunctions corresponding to the eigenvalues $e^{2\pi i/r^n}$, $n\geq1$, as a $\sigma$-algebra would have to be contained in the ``first coordinate'' which yields a contradiction by the following filtering type argument: the first coordinate $\sigma$-algebra together with the factor given by $\nu\ast\nu'$ generate the whole Borel structure of $X_f\times X_{f'}$. Indeed, if we take $x\in X_f$ and $x'\in X_{f'}$ and we know $x$ and $x\cdot x'$ then we also know $x'$ (we use here the fact that we consider sequences without zero elements). Similarly, we obtain a contradiction if $r=q$.

%

\subsection{Proof of Proposition~\ref{prop:periodic}}\label{s:periodic}

Proposition \ref{prop:periodic} is a direct consequence of the following, more precise statement.

\begin{Prop}\label{prop:periodic_precise}
Assume that $f$ is nonzero multiplicative and $M$-periodic. Let $\theta$ be its supporting character with the conductor $t$ (see Remark \ref{rem:pushout_nowa}). Then
$$
f(q^{b_q+\ell})=f(q^{b_q})\theta(q^{\ell})
$$
for every prime $q$ and $\ell\ge 0$, where $b_q:=\nu_q(M/t)$.
\end{Prop}

\begin{proof} We can assume that $M$ is the minimal period of $f$, then  $\spec(Per(f))=\spec(M)=\spec(m_1)$, where $m_1$ is the minimal period associated to the position 1 (see Lemma~\ref{lem:minperiod1}).
Let $q$ be a prime and assume that
 $\chi$ is the Dirichlet character of modulus $m_1$ such that $\chi(n)=f(n)$ for $n\in m_1^{\perp}=M^{\perp}$ as in Theorem \ref{th:chiporaz1}.

\underline{Case 1}: $q\nmid t$. In this case $\nu_q(M)=\nu_q(M/t)=b_q$, the number $M/q^{b_q}$ is coprime with $q$ and divisible by $t$.  Fix $\ell>0$. Clearly,  $q^{\ell}+\frac{M}{q^{b_q}}$ is coprime with $q$ and also belongs to $m_1^\perp$.
Then
$$
\begin{array}{l}
f(q^{b_q+\ell})=f(q^{b_q+\ell}+M)=f(q^{b_q}(q^{\ell}+\frac{M}{q^{b_q}}))=f(q^{b_q})f(q^{\ell}+\frac{M}{q^{b_q}})=\\
f(q^{b_q})\chi(q^{\ell}+\frac{M}{q^{b_q}})=f(q^{b_q})\theta(q^{\ell}+\frac{M}{q^{b_q}})=f(q^{b_q})\theta(q^{\ell}),
\end{array}
$$
the last equality follows since $t$ divides $M/q^{b_q}$.

\underline{Case 2}: $q|t$. Let $a=\nu_q(M)$ and $M'=\frac{M}{q^a}$. Then $M'$ is coprime with $q$. We divide the proof into three steps.

(i) We claim that $f(q^a)=0$. Indeed, suppose otherwise and denote by $R(n)$ the set
$$
(n+M'\N_0)\cap q^{\perp}.
$$
Since $M'$ is coprime with $q$, the set $R(n)$ is nonempty\footnote{Applying Dirichlet's theorem, consider $\gcd(n,M')x$ with $x$ prime sufficiently large and  $x\in\frac{n}{\gcd(n,M')}+\frac{M'}{\gcd(n,M')}\N$.} for every $n\in\N$. If $n+kM'\in R(n)$, then
$$
f(q^an)=f(q^an+kM)=f(q^a(n+kM'))=f(q^a)f(n+kM').
$$
Since we have supposed $f(q^a)\neq 0$, it follows that $f$ is constant on $R(n)$. Let $\psi(n)$ denote the value of $f$ on $R(n)$ for $n\in M'^{\perp}$.  Moreover, we set $\psi(n)=0$ for $n\notin M'^{\perp}$.  Observe that $f(n)=\psi(n)$ for $n\in M^{\perp}$ (since  $a\geq1$, so  $n\in R(n)$ in that case), so $\psi(n)=\theta(n)$ for $n\in M^{\perp}$. Moreover, as (for each $n$)
$$
R(n+M')\subset R(n)\text{ and } \spec(M')^{\perp}+M'=\spec(M')^{\perp},$$
it follows that $\psi$ is $M'$-periodic. We show now that $\psi$ is completely multiplicative.   Indeed,  let $n_1,n_2\in\N$. If one of $n_1,n_2$ has a divisor in $\spec(M')$, then $\psi(n_1)\psi(n_2)=0=\psi(n_1n_2)$. Otherwise (by applying Dirichlet's theorem), there exist $k_1,k_2\in\N$ such that $n_1+k_1M'\in R(n_1)$, $n_2+k_2M'\in R(n_2)$  and $\gcd(n_1+k_1M',n_2+k_2M')=1$. Then
$$
\psi(n_1)\psi(n_2)=f(n_1+k_1M')f(n_2+k_2M')=f(n_1n_2+kM')=\psi(n_1n_2),
$$
where $k=n_1k_2+n_2k_1+k_1k_2M'$.

Moreover,  by Lemma \ref{eq:periodof1}, $\psi(n)=f(n)\neq 0$ for $n\in M^{\perp}$, as $M^{\perp}=m_1^{\perp}$. Also, $q^j+M'\in R(q^j)$ and $\gcd(q^j+M',M)=1$ for every $j\ge 1$, so again, by  Lemma \ref{eq:periodof1}, we conclude that $\psi(q^j)=f(q^j+M')\neq 0$ for $j\ge 1$. Since $M=q^aM'$, this proves that $\psi(n)\neq 0$ if and only if $n\in M'^{\perp}$.
So $\psi$ is a Dirichlet character modulo $M'$.
We have $\psi(n)=\theta(n)=f(n)$ for $n\in M^{\perp}$, so $\psi$ induces $\chi$, so  $t|M'$ (see Remark \ref{rem:pushout_nowa}~(c)). This leads to a contradiction, since $q|t$ and $q\nmid M'$.
\medskip

(ii) Let $\ell>0$. Then, since $\gcd(q,q^{\ell}+M')=1$,
$$
f(q^{a+\ell})=f(q^{a+\ell}+M)=f(q^a(q^{\ell}+M'))=f(q^a)f(q^{\ell}+M')=0,
$$
the last equality follows from (i).
\medskip

(iii) Let $b_q<c<a$. Assume that $f(q^c)\neq 0$. Then, for every $k\in\N_0$, as $q|\frac{M}{q^c}$,
$$
f(q^c)=f(q^c+kM)=f(q^c(1+k\frac{M}{q^c}))=f(q^c)f(1+k\frac{M}{q^c}),
$$
so $\frac{M}{q^c}$ is a period associated to the position 1. Then $t|\frac{M}{q^c}$  (see Remark \ref{rem:pushout_nowa}~(c)), whence $\nu_q(t)\leq\nu_q(\frac{M}{q^c})$, while on the other side, because of the definition of $b_q$,
$$\nu_q(t)=\nu_q(M)-b_q>\nu_q(M)-c=\nu_q(M/q^c),$$  and we obtain a contradiction.
\medskip

We have proved that $f(q^{b_q+\ell})=0$ for every $\ell>0$, so $f(q^{b_q+\ell})=f(q^{b_q})\theta(q^{\ell})$, because $q$ divides the conductor $t$ of $\theta$, so $\theta(q)=0$.
\end{proof}

Let us finish with the following observation.

\begin{Prop}\label{p:odwrotny} Let $\theta$ be a Dirichlet character modulo $t$,  $q_1,\ldots,q_u$ - primes and $b_1,\ldots,b_u\in\N$. Assume that  $f$ is a multiplicative function such that $f(n)=\theta(n)$ for $n\in (q_1\ldots q_u)^{\perp}$ and
\begin{equation}\label{eq:p_odwrotny}
f(q_i^{b_i+\ell})=f(q_i^{b_i})\theta(q_i^{\ell})
\end{equation}
for every $i=1,\ldots,u$ and $\ell\ge 0$. Then $f$ is $q_1^{b_1}\ldots q_u^{b_u}t$-periodic.
\end{Prop}
\begin{proof}
We start with the following two easy observations:
\begin{enumerate}
\item[(a)] If $g$ is any multiplicative function such that $g(n)=1$ for $n\in (q_1\ldots q_u)^{\perp}$ and
$g(q_i^{b_i+\ell})=g(q_i^{b_i})$ for every $\ell\ge 0$ and $i=1,\ldots,u$, then $g$ is $q_1^{b_1}\ldots q_u^{b_u}$-periodic. Let us illustrate the main arguments in the situation $u=2$: let $n=q_1^{a_1}q_2^{a_2}n'$, where $n'$ is coprime with $q_1,q_2$. Assume that $a_1< b_1$, $a_2\ge b_2$. Then $g(n+q_1^{b_1}q_2^{b_2})=g(q_1^{a_1})g(q_2^{a_2}n'+q_1^{b_1-a_1}q_2^{b_2})$ as $k:=q_2^{a_2}n'+q_1^{b_1-a_1}q_2^{b_2}$ is coprime with $q_1$. Note that $q_2^{b_2}|k$, so $g(k)=g(q_2^{\nu_{q_2}(k)})g(k/q_2^{\nu_{q_2}(k)})=g(q_2^{\nu_{q_2}(k)})=g(q_2^{b_2})=g(q_2^{a_2})$. It follows that $g(n+q_1^{b_1}q_2^{b_2})=g(q_1^{a_1})g(q_2^{a_2})=g(n).$
\item[(b)] If $\theta$ is a Dirichlet character of modulus $t$ (or more generally, a periodic  multiplicative function of period $t$), $Q$ is some natural number and $\widetilde{\theta}$ is defined by
    $\widetilde{\theta}(n)=\theta(n/\gcd(n,Q))$, then $\widetilde{\theta}$ is multiplicative and periodic of period $Qt$.
\end{enumerate}
If $f$ satisfies the condition (\ref{eq:p_odwrotny}), then $f(n)=g(n)\widetilde{\theta}(n)$ for every $n$, where $g$ is defined by
$g(q_j^{k})=f(q^{\min(k,b_j)})$ for $k\ge 0$, $j=1,\ldots,u$ and $g(q)=1$ for primes $q\notin\{q_1,\ldots,q_u\}$ and $\widetilde{\theta}$ is defined as in (b) with $Q=q_1^{b_1}\ldots q_u^{b_u}$. Then $g$ is $Q$-periodic and $\widetilde{\theta}$ is $Qt$-periodic and  the statement follows, since the product of two multiplicative periodic functions with periods $m'$ and $m''$ is periodic with period $\lcm(m',m'')$.
\end{proof}

\subsection{Pretentious functions versus local 1-Fourier uniformity: proof of Proposition~\ref{p:nonl1fu}}\label{s:l1fu}
Following \cite{Ab-Ku-Le-Ru}, we say that a topological system $(X,T)$ satisfies the strong $f$-MOMO property if for each increasing sequence $(b_k)$ of natural numbers satisfying $b_{k+1}-b_k\to\infty$, we have
$$
\lim_{K\to\infty}\frac1{b_K}\sum_{k<K}\Big\|\sum_{b_k\leq n<b_{k+1}}f(n)\cdot F\circ T^n\Big\|_{C(X)}=0$$
for each $F\in C(X)$.
\begin{Lemma} \label{l:l1f} Assume that $f\in\mathcal{M}$ is pretentious. Then there exists $\alpha\in[0,1)$ such that $f$ does not satisfy the strong $f$-MOMO property for the rotation $Tx=x+\alpha$ on the additive circle $\T$.\end{Lemma}
\begin{proof} Assume that $\kappa\in V(f)$, that is, for some $(N_k)$, $\lim_{k\to\infty}\frac1{N_k}\sum_{n<N_k}\delta_{S^nf}=\kappa$. In view of \cite{Fr-Le-Ru}, the corresponding Furstenberg system $(X_f,\kappa,S)$ has discrete spectrum. It follows that there exists a compact Abelian group $G$ (considered with Haar measure $m_G$) such that for some $g_0\in G$, the rotation $Rx=x+g_0$ is ergodic and the Furstenberg system $(S,\kappa)$ is a factor of $R\times Id_{[0,1]}$ (on $[0,1]$ we consider Lebesgue measure ${\rm Leb}_{[0,1]}$), see e.g.\ \cite{Ed}.
Let
$$
p:(G\times [0,1],m_G\otimes {\rm Leb}_{[0,1]})\to (X_f,\kappa)$$
settles the factor map. Let $U_p$ stand for the corresponding Koopman operator:
$$
U_p:L^2(X_f,\kappa)\to
L^2(G\times [0,1],m_G\otimes {\rm Leb}_{[0,1]}).$$
Then $U_p$ yields a (graph) joining $\rho$ between the algebraic system $R\times Id_{[0,1]}$ and the Furstenberg system $S$, and, for some $\chi\in\widehat{G}$ and $j\in C([0,1])$, we have $U_p(\pi_0)\not\perp \chi\ot j$.
Therefore, using the definition of $\rho$ and the lifting lemma from \cite{Ka-Ku-Le-Ru} (see also Remark~1.4, Proposition~3.1 and Theorem~A (and its proof) in \cite{Ka-Ku-Le-Ru}), there exist a subsequence $(N_{k_\ell})$ of $(N_k)$ and a sequence $(g_n,t_n)\in G\times[0,1]$ such that the sequence $(g_n,t_n, S^nf)$ is quasi-generic (along $(N_{k_\ell})$) for $\rho$
and the set
$$\{b_1<b_2<\ldots\}:=\{n\colon (g_{n+1},t_{n+1})\neq (Rg_n,t_n)\}$$
has the property $b_{k+1}-b_k\to\infty$ and then (by slightly changing the sequence $(b_k)$ if necessary)
$$0\neq \int U_p(\pi_0)\cdot (\chi\ot j)\,d m_G\ot {\rm Leb}_{[0,1]}=
\int \pi_0\ot(\chi\ot j)\,d\rho=
$$
$$
\lim_{K\to\infty}\frac1{b_K}\sum_{k<K}\sum_{b_k\leq n<b_{k+1}}\pi_0(S^nf)\cdot (\chi\ot j)((R\times Id_{[0,1]})^n(g_k,t_k))=$$
$$
\lim_{K\to\infty}\frac1{b_K}\sum_{k<K}\sum_{b_k\leq n<b_{k+1}}u(n)\chi(g_k+ng_0)j(t_k)=$$$$
\lim_{K\to\infty}\frac1{b_K}\sum_{k<K}j(t_k)\chi(g_k)\sum_{b_k\leq n<b_{k+1}}(\chi(g_0))^n f(n).$$
It follows that
$$
\limsup_{K\to\infty}\frac1{b_K}\sum_{k<K}\Big|\sum_{b_k\leq n<b_{k+1}}(\chi(g_0))^nu(n)\Big|>0$$
and therefore (looking at the property of $(b_k)$), if $\chi(g_0)=e^{2\pi i \alpha}$, the rotation by $\alpha$
does not satisfy the strong $f$-MOMO property.
\end{proof}

In particular, we {\bf do not} have the following property: for each $(b_k)$ such that $b_{k+1}-b_k\to\infty$, we have
$$
(\ast)\;\;\lim_{K\to\infty}\frac1{b_K}\sum_{k<K}\Big\|\sum_{b_k\leq n<b_{k+1}}f(n)e^{2\pi in\cdot}\Big\|_{C(\T)}=0.$$

\noindent
{\em Proof of Proposition~\ref{p:nonl1fu}}
As shown in \cite{Ka-Le-Ri-Te}, the property $(\ast)$ is equivalent to~\eqref{l1fu1}, so the proposition follows from Lemma~\ref{l:l1f}.\bez

\vspace{1ex}
\noindent
Faculty of Mathematics and Computer Science, Nicolaus Copernicus University, Chopin street 12/18, 87-100 Toruń, Poland\\
skasjan@mat.umk.pl, mlem@mat.umk.pl

\vspace{2ex}
\noindent
School of Mathematics, University of Bristol, BS8 1QU, United
Kingdom\\ oleksiy.klurman@bristol.ac.uk

\end{document}